\documentclass[12pt]{amsart}

\usepackage{CJKutf8}

\usepackage{amssymb}
\usepackage{amsmath}
  \usepackage{paralist}
\usepackage{enumitem}
\usepackage{comment}

\usepackage[colorlinks,
          linkcolor = blue,		%%修改此处为你想要的颜色
          anchorcolor = blue,	%%修改此处为你想要的颜色
          urlcolor = red,		%%修改此处为你想要的颜色
          citecolor = red,		%%修改此处为你想要的颜色，例如修改blue为red
          ]{hyperref}
\newtheorem*{TheoremA'}{Theorem A'}

\newtheorem{thmm}{Theorem}

\newtheorem*{TheoremD'}{Theorem D'}

\newtheorem*{TheoremE'}{Theorem E'}

\newtheorem{theorem}{Theorem}[section]
\newtheorem{corollary}[theorem]{Corollary}

\newtheorem*{main*}{Main Theorem}
\newtheorem{lemma}[theorem]{Lemma}
\newtheorem{proposition}[theorem]{Proposition}

\newtheorem{question}[theorem]{Question}
\theoremstyle{definition}
\newtheorem{definition}[theorem]{Definition}
\newtheorem{remark}[theorem]{Remark}

\newcommand{\F}{\mathcal{F}}

\def\pX{{\partial X}}
\def\RR{{\mathbb R}}
\def\NN{{\mathbb N}}

\def\inj{{\text{inj}}}
\def\lx{{\underline x}}
\def\ly{{\underline y}}
\def\lv{{\underline v}}
\def\lc{{\underline c}}

\def\lV{{\underline V}}
\def\lN{{\underline N}}
\def\lm{{\underline m}}
\def\pr{{\text{pr}}}

\def\diam{\mathop{\hbox{{\rm diam}}}}

\def\diam{\mathop{\hbox{{\rm diam}}}}
\def\pr{\mathop{\hbox{{\rm pr}}}}

\def\top{{\mathop{\hbox{\footnotesize \rm top}}}}
\def\Reg{\mathop{\hbox{{\rm Reg}}}}
\def\Sing{\mathop{\hbox{{\rm Sing}}}}
\def\Rec{\mathop{\hbox{{\rm Rec}}}}

\def\a{\alpha}
\def\b{\beta}
\def\c{\gamma}   \def\C{\Gamma}
\def\d{\delta}   %\def\C{\Gamma}
   
 \def\e{\epsilon}

\def\ae{\text{-a.e.}\ }
\def\bP{\textbf{P}}
\def\bF{\textbf{F}}

\title[Running heading with forty characters or less]
      {Volume asymptotics and Margulis function in nonpositive curvature}

\author[first-name1 last-name1 and first-name2 last-name2]{Weisheng Wu}


\address{School of Mathematical Sciences, Xiamen University, Xiamen, 361005, P.R. China}
\email{wuweisheng@xmu.edu.cn}

\begin{document}

\maketitle
\markboth{Volume asymptotics and Margulis function}
{W. Wu}
\renewcommand{\sectionmark}[1]{}

\begin{abstract}
In this article, we consider a closed rank one $C^\infty$ Riemannian manifold $M$ of nonpositive curvature and its universal cover $X$. Let $b_t(x)$ be the Riemannian volume of the ball of radius $t>0$ around $x\in X$,  and $h$ the topological entropy of the geodesic flow. We obtain the following Margulis-type asymptotic estimates
\[\lim_{t\to \infty}b_t(x)/\frac{e^{ht}}{h}=c(x)\]
for some continuous function $c: X\to \RR$. We prove that the Margulis function $c(x)$ is in fact $C^1$. If $M$ is a surface of nonpositive curvature without flat strips, we show that $c(x)$ is constant if and only if $M$ has constant negative curvature. We also obtain a rigidity result related to the flip invariance of the Patterson-Sullivan measure.
\end{abstract}

\tableofcontents

\section{Introduction}
Consider a closed $C^\infty$ Riemannian manifold $(M,g)$ with negative sectional curvature everywhere. It is well known that the geodesic flow defined on the unit tangent bundle $SM$ is an Anosov flow (cf. \cite{An}, \cite[Section 17.6]{KH}). The ergodic theory of Anosov flows has many striking applications in the study of asymptotic geometry of the universal cover $X$ of $M$. In his celebrated 1970 thesis \cite{Mar1, Mar2}, Margulis obtained the following result:
\begin{equation}\label{e:mar}
\lim_{t\to \infty}b_t(x)/\frac{e^{ht}}{h}=c(x),
\end{equation}
where $b_t(x)$ is the Riemannian volume of the ball of radius $t>0$ around $x\in X$, $h$ the topological entropy of the geodesic flow, and $c: X\to \RR$ is a continuous function, which is called \emph{Margulis function}.

The main tool in the proof of Margulis's theorem is the Bowen-Margulis measure, which is the unique measure of maximal entropy (MME for short) for the Anosov flow (cf. \cite{Bo2}). Margulis \cite{Mar2} gave an explicit construction of this measure, and showed that it is mixing, and the conditional measures on stable/unstable manifolds have the scaling property, (i.e., contract/expand with a uniform rate under the geodesic flow), and are invariant under unstable/stable holonomies. Margulis \cite{Mar2} then proved \eqref{e:mar} using these ergodic properties of the Bowen-Margulis measure.

The ergodic theory of the geodesic flow on a closed rank one manifold of nonpositive curvature was developed by Pesin \cite{Pe1,Pe2} in 1970s. In this case the geodesic flow exhibits nonuniformly hyperbolic behavior (cf. \cite{BP}). In 1984, A.~Katok \cite{BuKa} conjectured that such geodesic flow also admits a unique MME. In 1998, Katok's conjecture was settled by Knieper \cite{Kn2}. In his proof, Knieper used Patterson-Sullivan measures on the boundary at infinity of the universal cover of $M$ to construct a MME (called \emph{Knieper measure}), and showed that this measure is the unique MME. Knieper \cite{Kn1} used his measure to obtain the following asymptotic estimates: there exists $C>0$ such that
$$\frac{1}{C}\le b_t(x)/e^{ht}\le C$$
for any $x\in X$.
However, it is difficult to improve the above to the Margulis-type asymptotic estimates \eqref{e:mar}, see the remark after \cite[Chapter 5, Theorem 3.1]{Kn3}. An unpublished preprint \cite{Gun} also contains many inspiring ideas to this problem. Recently, a breakthrough was made by Link (\cite[Theorem C]{Link}), where an asymptotic estimate for the orbit counting function is obtained for a CAT$(0)$ space, and as a consequence \eqref{e:mar} is established for rank one manifolds of nonpositve curvature.

A twin problem is the asymptotics of the number of free-homotopy classes of closed geodesics.  Margulis \cite{Mar1, Mar2} proved that in the negative curvature case
\begin{equation}\label{e:mar2}
\begin{aligned}
\lim_{t\to \infty}\#P(t)/\frac{e^{ht}}{ht}=1
\end{aligned}
\end{equation}
where $P(t)$ is the set of free-homotopy classes containing a closed geodesic with length at most $t$.
Recently Ricks \cite{Ri} proved \eqref{e:mar2} for rank one locally CAT$(0)$ spaces, which include rank one manifolds of nonpositve curvature. Later, generalizing Ricks' method, Climenhaga, Knieper and War \cite{CKW2} proved \eqref{e:mar2} for a class of manifolds (including all surfaces of genus at least $2$) without conjugate points, and the author \cite{Wu} proved \eqref{e:mar2} for rank one manifolds without focal points.

\subsection{Statement of main results}
In this paper, we first present an alternative approach to establishing volume asymptotics \eqref{e:mar} for rank one manifolds of nonpositive curvature. Then we study properties of the Margulis function and obtain related rigidity results.

Suppose that $(M,g)$ is a $C^{\infty}$ closed $n$-dimensional Riemannian manifold,
where $g$ is a Riemannian metric. Let $\pi: SM\to M$ be the unit tangent bundle over $M$. For each $v\in S_pM$,
we always denote by $c_{v}: \RR\to M$ the unique geodesic on $M$ satisfying the initial conditions $c_v(0)=p$ and $\dot c_v(0)=v$. The geodesic flow $\phi=(\phi^{t})_{t\in\mathbb{R}}$ (generated by the Riemannian metric $g$) on $SM$ is defined as:
\[
\phi^{t}: SM \rightarrow SM, \qquad (p,v) \mapsto
(c_{v}(t),\dot c_{v}(t)),\ \ \ \ \forall\ t\in \RR .
\]

A vector field $J(t)$ along a geodesic $c:\RR\to M$ is called a \emph{Jacobi field} if it satisfies the \emph{Jacobi equation}:
\[J''+R(J, \dot c)\dot c=0\]
where $R$ is the Riemannian curvature tensor and\ $'$\ denotes the covariant derivative along $c$.

A Jacobi field $J(t)$ along a geodesic $c(t)$ is called \emph{parallel} if $J'(t)=0$ for all $t\in \RR$. The notion of \emph{rank} is defined as follows.
\begin{definition}
For each $v \in SM$, we define \text{rank}($v$) to be the dimension of the vector space of parallel Jacobi fields along the geodesic $c_{v}$, and \text{rank}($M$):=$\min\{$\text{rank}$(v): v \in SM\}$. For a geodesic $c$ we define \text{rank}($c$):=\text{rank}($\dot c(t)$), $\forall\ t\in \mathbb{R}$.
\end{definition}

We let $M$ be a rank one closed Riemannian manifold of nonpositive curvature. Then $SM$ splits into two invariant subsets under the geodesic flow: the regular set $\Reg:= \{v\in SM: \text{rank}(v)=1\}$, and the singular set $\Sing:= SM \setminus \Reg$.

We have the following Margulis-type asymptotic estimates:
\begin{thmm}\label{margulis}
Let $M$ be a rank one closed Riemannian manifold of nonpositive curvature, $X$ the universal cover of $M$. Then
$$\lim_{t\to \infty}b_t(x)/\frac{e^{ht}}{h}=c(x),$$
where $b_t(x)$ is the Riemannian volume of the ball of radius $t>0$ around $x\in X$, $h$ the topological entropy of the geodesic flow, and $c: X\to \RR$ is a continuous function.
\end{thmm}

Theorem \ref{margulis} can be derived from \cite[Theorem C]{Link}, in which an asymptotic estimate for the orbit counting function was obtained for a rank one CAT$(0)$ space. Link's proof is quite geometric, and based on Roblin's method in negative curvature \cite{Rob}. We give an alternative proof in this paper using Margulis' approach which is more dynamical. Since we are dealing with manifolds, our proof can also be adapted to rank one manifolds without focal/conjugate points.

We use the notion of local product flow box and apply $\pi$-convergence theorem introduced by Ricks \cite{Ri} (see Section 2.4 and Section 3 below for more details). First we will establish the asymptotics formula for a pair of flow boxes (Subsections 4.1-4.4) using the mixing property of the unique MME and scaling property of Patterson-Sullivan measure. The asymptotics in \eqref{e:mar} involves countably many pairs of flow boxes and an issue of nonuniformity arises. To overcome this difficulty, we will apply Knieper's results and techniques (Lemmas \ref{singular} and \ref{nonuniform}).

In \cite{Ri, CKW2, Wu}, the authors also use flow boxes to calculate the asymptotic growth of the number of free homotopy classes of closed geodesics. Due to the equidistribution of closed geodesics, we only need consider one flow box and count the number of  self-intersections of the box under the geodesic flow. The technical novelties in this paper is that we have to deal with countably many pairs of flow boxes and the essential difficulty caused by nonuniform hyperbolicity. Even for one pair of flow boxes, the calculations involved are more complicated. Moreover we also consider the intersection components of a flow box with a face under the geodesic flow (see Subsection 4.1 below), which makes the calculation more subtle.

Manifolds without focal points/without conjugate points are natural generalizations of those of nonpositive curvature.
Theorem \ref{margulis} can be generalized to rank one manifolds without focal points and a certain class of manifolds without conjugate points as follows.
\begin{TheoremA'}\label{margulis1}
Let $M$ be either
\begin{enumerate}
  \item a rank one closed Riemannian manifold without focal points, or
  \item a closed manifold without conjugate points belonging to the class $\mathcal{H}$,
\end{enumerate}
and $X$ the universal cover of $M$.
Then
$$\lim_{t\to \infty}b_t(x)/\frac{e^{ht}}{h}=c(x),$$
where $b_t(x)$ is the Riemannian volume of the ball of radius $t>0$ around $x\in X$, $h$ the topological entropy of the geodesic flow, and $c: X\to \RR$ is a continuous function.
\end{TheoremA'}
We give the necessary definitions and discuss the proof of Theorem A' in the appendix. %The tools needed are already prepared in \cite{CKW2} and \cite{Wu}, so the proof can be modified accordingly.

%The main ingredients in the proof is the ergodic properties of Knieper measure and Patterson-Sullivan measures. The use of local product flow box and $\pi$-convergence theorem is crucial. We prepare these tools in Section 2.

The continuous function $c(x)$  is called \emph{Margulis function}. It is easy to see that
\begin{equation}\label{e:sphere}
\lim_{t\to \infty}s_t(x)/e^{ht}=c(x),
\end{equation}
where $s_t(x)$ is the spherical volume of the sphere $S(x,t)$ around $x\in X$ of radius $t>0$. It descends to a function on $M$, which we still denote by $c$.

In the negative curvature case, Katok conjectured that $c(x)$ is almost always not constant and not smooth. In \cite{Yue} Yue answered Katok's conjecture. He studied the uniqueness of harmonic measures associated to the strong stable foliation of the geodesic flow and obtained some rigidity results involving the Margulis function. We extend Yue's results to rank one manifolds of nonpositive curvature.
\begin{thmm}\label{function}
Let $M$ be a rank one closed $C^\infty$ Riemannian manifold of nonpositive curvature, $X$ the universal cover of $M$. Then
\begin{enumerate}
  \item The Margulis function $c$ is a $C^1$ function.
  \item If $c(x)\equiv C$, then for any $x\in X$,
  $$h=\int_{\partial X} tr U(x,\xi)d\tilde \mu_x(\xi)$$
where $U(x,\xi)$ and $tr U(x,\xi)$ are the second fundamental form and the mean curvature of the horosphere $H_x(\xi)$ respectively, and $\tilde \mu_x$ is the normalized Patterson-Sullivan measure.
\end{enumerate}
\end{thmm}

To the best of our knowledge, the uniqueness of harmonic measures in nonpositive curvature is not known.
\begin{question}
For manifolds of nonpositive curvature, do we have a unique harmonic measure associated to the strong stable foliation of the geodesic flow?
Do the leaves of the strong stable foliation have polynomial growth?
\end{question}
\begin{remark}
For rank one manifolds without focal points, the uniqueness of harmonic measures associated to the weak stable foliation of the geodesic flow is proved in \cite[Theorem 3.1]{LS}.
\end{remark}

If $\dim M=2$, then $\text{Vol}(B^s(x,r))=2r$ where $B^s(x,r)$ is any ball of radius $r$ in a strong stable manifold. In this case, $B^s(x,r)$ is just a curve. Hence the leaves of the strong stable foliation have polynomial growth. Combining with a recent result in \cite{Cl} on the unique ergodicity of the horocycle flow, we can show that there is a unique harmonic measure associated to the strong stable foliation for rank one surfaces of nonpositive curvature without flat strips. Then we can prove the following rigidity result.
\begin{thmm}\label{rigidity}
Let $M$ be a rank one closed Riemannian surface of nonpositive curvature without flat strips. Then $c(x)\equiv C$ if and only if $M$ has constant negative curvature.
\end{thmm}

Without the uniqueness of harmonic measures, we can still obtain some rigidity results in arbitrary dimension. The flip map $F:SM\to SM$ is defined as $F(v):=-v$. By the construction, the Knieper measure $\lm$ is flip invariant. Consider the conditional measures $\{\bar\mu_x\}_{x\in M}$ of $\lm$ with respect to the partition $SM=\cup_{x\in M}S_xM$. $\bar\mu_x$ can be identified as measures on $\partial X$, and it would be natural to consider if $\bar\mu_x$ and the normalized Patterson-Sullivan measures $\tilde\mu_x$ coincide.
Yue \cite{Yue1, Yue2} obtained related rigidity results in negative curvature, which can be generalized to the nonpositive curvature case with the help of the Margulis function.
\begin{thmm}\label{flip}
The conditional measures $\{\bar\mu_x\}_{x\in M}$ of the Knieper measure coincide almost everywhere with the normalized Patterson-Sullivan measures $\tilde\mu_x$ if and only if $M$ is locally symmetric.
\end{thmm}

\section{Geometric and ergodic toolbox}
We prepare some geometric and ergodic tools for rank one manifolds of nonpositive curvature, which will be used in subsequent sections.

\subsection{Boundary at infinity}
Let $M$ be a closed Riemannian manifold of nonpositive curvature, and $\pr: X\to M$ the universal cover of $M$. Let $\C \simeq \pi_1(M)$ be the group of deck transformations on $X$, so that each $\c\in \C$ acts isometrically on $X$. Let $\F$ be a fundamental domain with respect to $\C$.
%Since $M=X/\Gamma$ is compact, each $\c\in \Gamma$ is \emph{axial} (cf. \cite{CS} Lemma 2.1), that is, there exists a geodesic $c$ and $t_{0}>0$ such that $\c(c(t)) = c(t+t_{0})$ for every $t\in \mathbb{R}$. Correspondingly $c$ is called an \emph{axis} of $\c$ and we denote $|\c|:=t_0$ where $t_0$ is minimal with the above property.
Denote by $d$ both the distance functions on $M$ and $X$ induced by Riemannian metrics. The Sasaki metrics on $SM$ and $SX$ are also denoted by $d$ if there is no confusion.

We still denote by $\pr: SX\to SM$ and $\c: SX\to SX$ the map on unit tangent bundles induced by $\pr$ and $\c\in \C$. From now on, we use an underline to denote objects in $M$ and $SM$, e.g. for a geodesic $c$ in $X$ and $v\in SX$, $\lc:=\pr c$, $\lv:=\pr v$ denote their projections to $M$ and $SM$ respectively.

One consequence of nonpositive curvature is the convexity of the distance functions such as $t\mapsto d(c_{1}(t),c_{2}(t))$ and $t\mapsto d(q,c_{1}(t))$ where $c_{1}$ and $c_{2}$ are two geodesics in $X$ and $q\in X$. Another frequently used property is the comparison theorem: for every triple of distinct points $x, y, z \in X$, the geodesic triangle is no fatter than the corresponding comparison triangle in the Euclidean plane $\RR^2$ with the same edge lengths. As a consequence, if $\triangle=\triangle(x_1,x_2,x_3)$ is a triangle in $X$ and $\bar\triangle=\triangle(\bar x_1,\bar x_2,\bar x_3)$ is the comparison triangle in $\RR^2$ (i.e., $d(x_i,x_j)=d(\bar x_i,\bar x_j), i,j=1,2,3$), then
\begin{equation}\label{e:comp}
\a_i\le \bar\a_i,
\end{equation}
where $\a_i$ are angles of $\triangle$ at $x_i$ and $\bar \a_i$ is the corresponding angles in $\bar x_i$.
See \cite{Bal} for more details.

We call two geodesics $c_{1}$ and $c_{2}$ on $X$ \emph{positively asymptotic} or just \emph{asymptotic} if there is a positive number $C > 0$ such that $d(c_{1}(t),c_{2}(t)) \leq C, ~~\forall~ t \geq 0.$
%We say $c_{1}$ and $c_{2}$ are \emph{negatively asymptotic} if \eqref{e1} holds for all $t \leq 0$. $c_{1}$ and $c_{2}$ are said to be \emph{bi-asymptotic} or \emph{parallel} if they are both positively and negatively asymptotic.
The relation of asymptoticity is an equivalence relation between geodesics on $X$. The class of geodesics that are asymptotic to a given geodesic $c_v/c_{-v}$ is denoted by $c_v(+\infty)$/$c_v(-\infty)$ or $v^+/v^-$ respectively. We call them \emph{points at infinity}. Obviously, $c_{v}(-\infty)=c_{-v}(+\infty)$. We call the set $\pX$ of all points at infinity the \emph{boundary at infinity}. If $\eta=v^+\in \pX$, we say $v$ \emph{points at $\eta.$}

We can define the visual topology on $\pX$ following \cite{Eb1} and \cite{EO}. For each $p$, there is a bijection $f_p: S_pX\to \pX$ defined by $$f_p(v)=v^+, v\in S_pX.$$
So for each $p\in M$, $f_p$ induces a topology on $\pX$ from the usual topology on $S_pX$.
%Given $p,q\in X$, let $\psi: S_pX\to S_qX$ be the map such that $\psi(v)$ is the unique vector in $S_qX$ asymptotic to $v\in S_pX$.
%$\psi$ is in fact a homoeomorphism.
The topology on $\pX$ induced by $f_p$ is independent of $p\in X$, and is called the \emph{visual topology} on $\pX$.

Visual topology on $\pX$ and the manifold topology on $X$ can be extended naturally to the so-called \emph{cone topology} on $\overline X:= X\cup \pX$.
 %by requiring the map $\varphi$ defined as follows is a homeomorphism. Fix $p\in X$. For each $v\in T_pX$ with $\|v\|\le 1$, define
%\[
%\varphi(v):=\begin{cases}\exp (\frac{v}{1-\|v\|}) &
 %  \text{if} \ \|v\|<1; \\
%f_p(v) &\text{if} \ \|v\|=1.
%\end{cases}
%\]
Under cone topology, $\overline{X}$ is homeomorphic to the closed unit ball in $\mathbb{R}^{n}$, and $\pX$ is homeomorphic to the unit sphere $\mathbb{S}^{n-1}$. For $x,y\in \overline{X}$, we denote by $c_{x, y}$ the geodesic connection $x$ and $y$ if it exists.

The \emph{angle metric} on $\partial X$ is defined as
$$\angle(\xi,\eta):=\sup_{x\in X}\angle_x(\xi,\eta), \quad \forall \xi, \eta\in \pX,$$
where $\angle_x(\xi,\eta)$ denotes the angle between the unit tangent vectors at $x$ of the geodesics $c_{x,\xi}$ and $c_{x,\eta}$.
Then the angle metric defines a path metric $d_T$ on $\pX$, called the \emph{Tits metric}. More precisely, for a continuous curve $c:[0,1]\to \pX$, define the length
$L(c):=\sum_{i=0}^{k-1}\angle(c(t_i), c(t_{i+1}))$ where $0=t_0\le t_1\le \cdots \le t_k=1$ is a subdivision of $[0,1]$. Then we can define the path metric $d_T(\xi,\eta):=\inf L(c)$, where the infimum is taken over all the continuous curves joining $\xi$ and $\eta$. Clearly, $d_T\ge \angle$.
The angle metric induces a topology on $\partial X$ finer than the visual topology. Moreover, $v\in SX$ is regular if and only if that $d_T(v^-,v^+)>\pi$. In particular, if $M$ has negative curvature everywhere, then $\angle(\xi,\eta)=\pi$ and hence $d_T(\xi,\eta)=\infty$ for any $\xi\neq \eta\in \pX$. See \cite{BGS} for more information on Tits metric.

\subsection{Busemann function}
For each pair of points $(p,q)\in X \times X$ and each point at infinity $\xi \in \pX$, the \emph{Busemann function based at $\xi$ and normalized by $p$} is
$$b_{\xi}(q,p):=\lim_{t\rightarrow +\infty}\big(d(q,c_{p,\xi}(t))-t\big),$$
where $c_{p,\xi}$ is the unique geodesic from $p$ and pointing at $\xi$.
The Busemann function $b_{\xi}(q,p)$ is well-defined since the function $t \mapsto d(q,c_{p,\xi}(t))-t$
is bounded from above by $d(p,q)$, and decreasing in $t$ (this can be checked by using the triangle inequality). Obviously, we have
$$|b_{\xi}(q,p)|\le d(q,p).$$
If $v\in S_pX$ points at $\xi\in \pX$, we also write $b_v(q):=b_{\xi}(q,p).$

The level sets of the Busemann function $b_{\xi}(q,p)$ are called the \emph{horospheres} centered at $\xi$. The horosphere through $p$ based at $\xi\in \pX$, is denoted by $H_p(\xi)$.
For more details of the Busemann functions and horospheres, please see ~\cite{DPS,Ru1,Ru2}.
%Here we are concerned with the equicontinuity property of Busemann functions.

%We say that the manifold $M$ satisfies the \emph{axiom of asymptoticity} if for every $v\in SX$
%the following statement holds:
%For any choice of $x_n, x\in X, v_n\in SX$, $x_n\to x$, $v_n\to v$ and a sequence of numbers $t_n\to \infty$, the sequence $c_n$ of geodesic segments joining $x_n$ to $c_{v_n}(t_n)$ converges to an asymptote of $c_v$.

According to \cite[Theorem 6.1]{Pe2} and \cite[Lemma 1.2]{Ru1}, we have the following continuity property of Busemann functions.
%\label{horosphere}
%Let $M$ be a closed Riemannian manifold without conjugate points that satisfies the the axiom of asymptoticity. Then for every $p\in X$, the map $\xi\mapsto H_\xi(p)$ is continuous in the following sense: if $\xi_n\to \xi \in \pX$ and $K\subset X$ is compact, then $H_{\xi_n}(p)\cap K\to H_\xi(p)\cap K$ uniformly in the Hausdorff topology.
%\end{proposition}

%\begin{proposition}(Cf. \cite[Theorem 5.2]{Pe2})\label{axiom}
%If the manifold $M$ has no focal points, then it satisfies the the axiom of asymptoticity.
%\end{proposition}
%By Proposition \ref{axiom}, Proposition \ref{horosphere} applies to manifolds without focal points. So we have the following corollary
\begin{lemma}\label{continuous} (Cf. \cite[Corollary 2.7]{Wu}):
The functions $(v,q)\mapsto b_v(q)$ and $(\xi,p,q)\mapsto b_\xi(p,q)$ are continuous on $SX\times X$ and $\pX\times X\times X$ respectively.
\end{lemma}

In fact, we have the following equicontinuity property of Busemann function $v\mapsto b_v(q)$.
\begin{lemma}\label{equicon1}(Cf. \cite[Lemma 2.9]{Wu})
Let $p\in X$, $A\subset S_pX$ be closed, and $B\subset X$ be such
that $A^+:= \{v^+: v\in A\}$ and $B^\infty := \{\lim_n q_n \in \pX: q_n\in B\}$ are
disjoint subsets of $\pX$. Then the family of functions $A\to \RR$ indexed
by $B$ and given by $v\mapsto b_v(q)$ are equicontinuous: for every $\e>0$ there
exists $\d >0$ such that if $\angle_p(v,w)<\d$, then $|b_v(q)-b_w(q)|<\e$ for every $q\in B$.
\end{lemma}

\subsection{Patterson-Sullivan measure and Knieper measure}
We will recall the construction of the Patterson-Sullivan measure and the Knieper measure, which are the main tools to the subsequent proofs.
\begin{definition}\label{density}
Let $X$ be a simply connected manifold of nonpositive curvature and
$\Gamma$ a discrete subgroup of $\text{Iso}(X)$, the group of isometries of $X$. For a given constant $r>0$, a family of finite Borel measures $\{\mu_{p}\}_{p\in X}$
on $\pX$ is called an $r$-dimensional \emph{Busemann density} if
\begin{enumerate}
\item for any $p,\ q \in X$ and $\mu_{p}$-a.e. $\xi\in \pX$,
$$\frac{d\mu_{q}}{d \mu_{p}}(\xi)=e^{-r \cdot b_{\xi}(q,p)}$$
 where $b_{\xi}(q,p)$ is the Busemann function;
\item $\{\mu_{p}\}_{p\in X}$ is $\Gamma$-equivariant, i.e., for all Borel sets $A \subset \pX$ and for any $\c \in \Gamma$,
we have $$\mu_{\c p}(\c A) = \mu_{p}(A).$$
\end{enumerate}
\end{definition}
The construction of such a Busemann density is due to \cite{Kn1} via Patterson-Sullivan construction. Moreover, it is showed in \cite[Theorem 4.10]{Kn1} that up to a multiplicative constant, the Busemann density is unique, i.e., the Patterson-Sullivan measure is the unique Busemann density.

The following Shadowing Lemma is one of the most crucial properties of the Patterson-Sullivan measure.
\begin{lemma} (Shadowing Lemma, cf. \cite[Chapter 5, Proposition 2.1]{Kn3})\label{shadow}
Let $\{\mu_{p}\}_{p\in X}$  be the Patterson-Sullivan measure, which is the unique Busemann density with dimension $h$. Then there exists $R>0$ such that for any $\rho\ge R$ and any $p,x\in X$ there is $b=b(\rho)$ with
\begin{equation*}
\begin{aligned}
b^{-1}e^{-hd(p,x)}\le \mu_{p}(\bar f_p(B(x,\rho)))\le be^{-hd(p,x)}
\end{aligned}
\end{equation*}
where $\bar f_p(y):=c_{p,y}(+\infty)$ for any $y\in B(x,\rho)$.
\end{lemma}

Let $P:SX \to \pX \times \pX$ be the projection given by $P(v)=(v^-,v^+)$. Denote by $\mathcal{I}^{P}:=P(SX)=\{P(v)\mid v \in SX\}$ the subset of pairs in $\pX \times \pX$
which can be connected by a geodesic. Note that the connecting geodesic may not be unique and moreover, not every pair $\xi\neq \eta$ in $\pX$ can be connected by a geodesic.

Fix a point $p\in X$, we can define a $\Gamma$-invariant measure $\overline{\mu}$
on $\mathcal{I}^{P}$ by the following formula:
\begin{equation*}
d \overline{\mu}(\xi,\eta) := e^{h\cdot \beta_{p}(\xi,\eta)}d\mu_{p}(\xi) d\mu_{p}(\eta),
\end{equation*}
where $\beta_{p}(\xi,\eta):=-\{b_{\xi}(q, p)+b_{\eta}(q, p)\}$ is the Gromov product, and $q$ is any point on a geodesic $c$ connecting $\xi$ and $\eta$. By \cite[Lemma 2.1]{Kn2} (see also Proposition \ref{crucial} below), $\overline{\mu}(\mathcal{I}^{P})>0$. It is easy to see that the function $\beta_{p}(\xi,\eta)$ does not depend on the choice of $c$ and $q$.
In geometric language, the Gromov product $\beta_{p}(\xi,\eta)$ is the length of the part of a geodesic
$c$ between the horospheres $H_{\xi}(p)$ and $H_{\eta}(p)$.

Then $\overline{\mu}$ induces a $\phi^t$-invariant measure $m$ on $SX$ with
\begin{equation*}
m(A)=\int_{\mathcal{I}^{P}} \text{Vol}\{\pi(P^{-1}(\xi,\eta)\cap A)\}d \overline{\mu}(\xi,\eta),
\end{equation*}
for all Borel sets $A\subset SX$. Here $\pi : SX \rightarrow X$ is the standard projection map and
Vol is the induced volume form on $\pi(P^{-1}(\xi,\eta))$.
%By the definition of $P$, we know that if there is no geodesic connecting $\xi$ and $\eta$,
%then $P^{-1}(\xi,\eta)=\emptyset$.
If there are more than one geodesics connecting $\xi$ and $\eta$, then by the flat strip theorem, $\pi(P^{-1}(\xi,\eta))$ is exactly a $k$-flat submanifold connecting $\xi$ and $\eta$ for some $k\ge 2$,
which consists of all the geodesics connecting $\xi$ and $\eta$.
%Especially, when $k$=1, $P^{-1}(\xi,\eta)$ is exactly the rank 1 (thus unique) geodesic connecting $\xi$ and $\eta$.

For any Borel set $A\subset SX$ and $t\in \mathbb{R}$,
$\text{Vol}\{\pi(P^{-1}(\xi,\eta)\cap \phi^{t}A)\}=\text{Vol}\{\pi(P^{-1}(\xi,\eta)\cap A)\}$.
Therefore $m$ is $\phi^t$-invariant. Moreover, $\Gamma$-invariance of $\overline{\mu}$ leads to the $\Gamma$-invariance of $m$.
Thus $m$ induced a $\phi^t$-invariant measure $\lm$ on $SM$ which is determined by
\begin{equation*}
m(A)=\int_{SM} \#(\text{pr}^{-1}(\lv)\cap A)d\lm(\lv).
\end{equation*}
It is proved in \cite{Kn2} that $\lm$ is unique MME, which is called Knieper measure.
Furthermore, $\lm$ is proved to be mixing in \cite{Ba} and eventually Bernoulli in \cite{CT}.

\subsection{Local product flow boxes}
In this subsection, we fix a regular vector $v_0\in SX$.  Let $p:=\pi(v_0)$, which will be the reference point in the following discussions. %For simplicity, we will suppress $v_0$ and $p$ from the notation.
We also fix a scale $\e\in (0, \min\{\frac{1}{8}, \frac{\inj (M)}{4}\})$.

The \emph{Hopf map} $H: SX\to \pX\times \pX\times \RR$ for $p\in X$ is defined as
$$H(v):=(v^-, v^+, s(v)), \text{\ where\ }s(v):=b_{v^-}(\pi v, p).$$
From definition, we see
$s(\phi^t v)=s(v)+t$
for any $v\in SX$ and $t\in \RR$. $s$ is continuous by Lemma \ref{continuous}.

Following \cite{CKW2} and \cite{Wu}, we define for each $\theta>0$ and $0<\a<\frac{3}{2}\e$,
\begin{equation*}
\begin{aligned}
&\bP_\theta:=\{w^-: w\in S_pX \text{\ and\ }\angle_p(w, v_0)\le \theta\},\\
&\bF_\theta:=\{w^+: w\in S_pX \text{\ and\ }\angle_p(w, v_0)\le \theta\},\\
&B_\theta^\a:=H^{-1}(\bP_\theta\times \bF_\theta\times [-\a,\a]).
\end{aligned}
\end{equation*}
$B_\theta^\a$ is called a \emph{flow box} with depth $\a$. We will consider $\theta>0$ small enough, which will be specified in the following.

The following lemma was crucial in constructing Knieper measure.
\begin{proposition}\label{crucial}(\cite[Lemma 2.1]{Kn2})
Let $X$ be a simply connected manifold of nonpositive curvature and $v_0\in SX$ is regular. Then for any $\e>0$, there is an $\theta_1 > 0$ such that, for any $\xi \in \bP_{\theta_1}$ and $\eta \in \bF_{\theta_1}$,
there is a unique geodesic $c_{\xi,\eta}$ connecting $\xi$ and $\eta$, i.e.,
$c_{\xi,\eta}(-\infty)=\xi$ and $c_{\xi,\eta}(+\infty)=\eta$.

Moreover, the geodesic $c_{\xi,\eta}$ is regular and $d(c_{v}(0), c_{\xi,\eta})<\epsilon/10$.
\end{proposition}

Based on Proposition \ref{crucial}, we have the following result.
\begin{lemma}\label{diameter}(\cite[Lemma 2.13]{Wu})
Let $v_0, p, \e$ be as above and $\theta_1$ be given in Proposition \ref{crucial}. Then for any $0<\theta<\theta_1$,
\begin{enumerate}
  \item $\text{diam\ } \pi H^{-1}(\bP_\theta\times \bF_\theta\times \{0\})<\frac{\e}{2}$;
  \item $H^{-1}(\bP_\theta\times \bF_\theta\times \{0\})\subset SX$ is compact;
  \item $\text{diam\ } \pi B_\theta^\a <4\e$ for any $0<\a\le \frac{3\e}{2}$.
\end{enumerate}
\end{lemma}
\begin{comment}
At last, we also assume that the choice of $\theta\in (0, \theta_1)$ satisfies the following properties. Since $\theta\mapsto \bar\mu(\bP_\theta\times \bF_\theta)$ is nondecreasing and hence has at most countably many discontinuities, we always choose $\theta$ to be the continuity point of this function, i.e.,
$$\lim_{\rho\to \theta}\bar\mu(\bP_\rho\times \bF_\rho)=\bar\mu(\bP_\theta\times \bF_\theta).$$
Furthermore, $\theta$ is chosen so that $\bar\mu(\partial\bP_\theta\times \partial\bF_\theta)=0.$ By the product structure of $B$ and $S$ and the definition of $m$, we have for any $\a\in (0,\frac{3\e}{2})$,
\begin{equation}\label{e:choice}
\begin{aligned}
\lim_{\rho\to \theta} m(S_\rho)=m(S_\theta), \quad \lim_{\rho\to \theta} m(B^\a_\rho)=m(B^\a_\theta),
\text{\ \ and\ \ }m(\partial B^\a_\theta)=0.
\end{aligned}
\end{equation}
\end{comment}
The following is a direct corollary of Lemma \ref{equicon1}.
\begin{corollary}\label{equicon}
Given $v_0, p, \e>0$ as above, there exists $\theta_2>0$ such that for any $0<\theta<\theta_2$, if $\xi,\eta\in \bP_\theta$ and any $q$ lying within $\diam \F+4\e$ of $\pi H^{-1}(\bP_\theta\times \bF_\theta\times [0,\infty))$, we have $|b_\xi(q,p)-b_\eta(q,p)|<\e^2$. Similar result holds if the roles of $\bP_\theta$ and $\bF_\theta$ are reversed.
\end{corollary}
Let $\theta_0:=\min\{\theta_1, \theta_2\}$, where $\theta_1$ is given in Lemma \ref{diameter}, and $\theta_2$ is given in Corollary \ref{equicon}. In the following, we always suppose that $0<\theta<\theta_0$.

\subsection{Regular partition-cover}
Let us fix $x,y\in \F\subset X$, and $p=x$ the reference point. For each regular vector $w\in S_{x}X$, we can construct a local product flow box around $w$ as in the last subsection. More precisely, consider the interior of $B_\theta^\a(w)$, $\text{int} B_\theta^\a(w)$, which is an open neighborhood of $w$ for some $\a>0$ and $0<\theta<\theta_0$ (here $\theta_0$ depends on $w$). By second countability of $S_xX$, there exist countably many regular vectors $w_1, w_2, \cdots$ such that $S_xX\cap \Reg \subset \cup_{i=1}^\infty\text{int} B_{\theta_i}^\a(w_i)$.
Similarly, there exist countably many regular vectors $v_1, v_2, \cdots$ such that $S_yX\cap \Reg \subset \cup_{i=1}^\infty\text{int} B_{\theta'_i}^\a(v_i)$. We note that the reference point $p$ is always chosen to be $x$ in the construction of all above flow boxes.

A \emph{regular partition-cover} of $S_xX$ is a triple $(\{w_i\}, \{\text{int} B_{\theta_i}^\a(w_i)\}, \{N_i\})$ where $\{N_i\}$ is a disjoint partition of $\Reg\cap S_xX$ and such that $N_i\subset \text{int} B_{\theta_i}^\a(w_i)$ for each $i\in\NN$. Similarly a regular partition-cover of $S_yX$ is a triple $(\{v_i\}, \{\text{int} B_{\theta'_i}^\a(v_i)\}, \{V_i\})$ such that  $\{V_i\}$ is a disjoint partition of $\Reg\cap S_yX$ and $V_i\subset \text{int} B_{\theta'_i}^\a(v_i)$ for each $i\in\NN$.

Recall the bijection $f_x: S_xX\to \pX$ defined by $f_x(v)=v^+, v\in S_xX$. Let $\tilde \mu_x:=(f_x^{-1})_*\mu_p$ which is a finite Borel measure on $S_xX$. Similarly, let $\tilde \mu_y:=(f_y^{-1})_*\mu_p$ be a finite Borel measure supported on $S_yX$. The measures $\tilde \mu_x$ and $\tilde \mu_y$ will be used in Section 4.

The following result is essentially proved in \cite[Proposition 2.4]{Cl}.
\begin{lemma}\label{null}
For any $x\in X$, we have $\tilde \mu_x(\Reg\cap S_xX)=\tilde \mu_x(S_xX)$.
\end{lemma}
\begin{proof}
Define $\Sing^+:=\{v^+: v\in \Sing\}$. By \cite[Proposition 4.9]{Kn1}, $\C$ acts ergodically on $\partial X$. Since $\Sing^+$ is $\C$-invariant, we have that either $\mu_x(\Sing^+)=\mu_x(\pX)$ or $\mu_x(\Sing^+)=0$ for any $x\in X$.

Let $\Rec\subset SM$ be the subset of vectors recurrent under the geodesic flow. Then its lift to $SX$, which is also denoted by $\Rec$, has full $m$-measure. By \cite{Kn2}, $\Reg$ also has full $m$-measure, and thus $\Rec\cap \Reg$ has full $m$-measure. Define $R:=\{v^+: v\in \Rec\cap \Reg\}$. By definition of the Knieper measure, we see that $R$ has full $\mu_x$-measure.

Let $v\in \Rec\cap \Reg$. By \cite[Proposition 4.1]{Kn2}, for every
$$w\in W^s(v):=\{w\in SX: w=-\text{grad}\ b_\xi(q, \pi v), b_\xi(q, \pi v)=0\},$$
i.e., every vector $w$ on the strong stable horocycle manifold of $v$, we have $d(\phi^tv, \phi^tw)\to 0$ as $t\to +\infty$. Since $v$ is recurrent and regular, then $w$ is also regular. It follows that $R\cap \Sing^+=\emptyset$ and so $\mu_x(\Sing^+)=0$. The lemma then follows.
\end{proof}

From now on in Sections 3-4, we will first consider a pair of $V_i$ and $N_j$ from the regular partition-covers of $S_xX$ and $S_yX$ respectively. At last, we will sum up our estimates over countably many such pairs from  regular partition-covers.

\begin{comment}
Given $\lx,\ly\in M$, consider $\lN=S_{\lx}M, \lV=S_{\ly}M$ and their regular partition-covers, $(\lv_i, \text{int} B_\theta^\a(\lv_i), \lN_i)$ and $(\lv_i, \text{int} B_\theta^\a(\lv_i), \lV_i)$.

Now let $x,y$ be the lift of $\lx,\ly$ on $X$ such that $x,y\in \F$ where $\F$ is a fundamental domain of $X$.
Let $p\in \F$ be a reference point. We lift the above objects to $\F$, and get $N=S_{x}X, V=S_{y}X$ and regular partition-covers, $(v_i, \text{int} B_\theta^\a(v_i), N_i)$ and $(v_i, \text{int} B_\theta^\a(v_i), V_i)$.

Let $S\subset S_xX\times S_yX$ be the set of all pairs of vectors $(v,w)$ with $v\in S_xX$ and $w\in S_yX$ satisfying one of the following conditions:
\begin{enumerate}
  \item $v$ is singular;
  \item $w$ is singular;
  \item if $c_{w^-,v^+}$ is a geodesic connecting $w^-$ and $v^+$, then $c_{w^-,v^+}$ is not the boundary of a half plane.
\end{enumerate}
\begin{lemma}
We have $(\tilde \mu_x\times\tilde \mu_y)(S)=0$ for $(\pi_*m\times \pi_*m)\ae x,y \in X$.
\end{lemma}
\end{comment}

\section{Local uniform expansion}
In this section, we use $\pi$-convergence theorem to illustrate local uniform expansion along unstable horospheres. As a consequence, we obtain estimates on the cardinality of certain subgroups of $\C$.

Let $B_T(\xi,r)$ denote the open ball $\{\eta\in \pX: d_T(\eta,\xi)< r\}$ and $\overline{B}_T(\xi,r)$ the closed ball $\{\eta\in \pX: d_T(\eta,\xi)\le r\}$. Recall that $v\in SX$ is regular if and only if that $d_T(v^-,v^+)>\pi$.

\begin{theorem}($\pi$-convergence theorem, \cite[Lemma 18]{PS})\label{piconvergence}
Let $X$ be a simply connected manifold of nonpositive curvature and $\C$ a group acting by isometries on $X$. %$v_0, p, \e$ be fixed as in Section $2.4$, and $\theta_1$ be given in Proposition \ref{crucial}. Fix any $0<\rho<\theta<\theta_1$.
Let $x\in X$, $\c_i\in \Gamma$ and $\theta\in [0,\pi]$, and suppose that $\c_i(x)\to \xi\in \pX$ and $\c^{-1}_i (x)\to \eta \in \pX$ as $i\to \infty$. Then for any compact set $K\subset \pX\setminus \overline{B}_T(\eta,\theta)$, we have $\c_i(K)\to \overline{B}_T(\xi,\pi-\theta)$, in the sense that for any open set $U$ with $U\supset \overline{B}_T(\xi,\pi-\theta)$, $\c_i(K) \subset U$ for all $i$ sufficiently large.
\end{theorem}

Consider a pair of $V_i$ and $N_j$ from the regular partition-covers of $S_xX$ and $S_yX$ respectively. For simplicity, we just denote $V:=V_i$ and $N:=N_j$. Then $N\subset \text{int}B_{\theta_j}^\a(w_0)$ and $V\subset \text{int}B_{\theta'_i}^\b(v_0)$ for some $w_0\in S_xX$ and $v_0\in S_yX$. We also denote for $a>0$,
\begin{equation*}
\begin{aligned}
B_aN&:=\{v\in S_xX: d(v,N)\le a\}\\
B_{-a}N&:=\{v\in N: B(v,a)\subset N\}.
\end{aligned}
\end{equation*}
Write $t_0:=s(v_0)=b_{v_0^-}(\pi v_0,p)$ where $p=x$. We denote the flow boxes by
\begin{equation*}
\begin{aligned}
&N^{\a}:=H^{-1}(N^-\times N^+\times [-\a,\a]),\\
&V^\b:=H^{-1}(V^-\times V^+\times (t_0+[-\b,\b])).
\end{aligned}
\end{equation*}

Notice that $N^\a\subset \text{int}B_{\theta_j}^\a(w_0)$ and $V^\b\subset \text{int}B_{\theta'_i}^\b(v_0)$.
Given $\e>0$, we always consider $\frac{\e^2}{100}\le \a, \b\le \frac{3\e}{2}$.
By carefully adjusting the regular partition-covers, we can guarantee that
\begin{equation}\label{e:boundary}
\begin{aligned}
\mu_p(\partial V^+)=\mu_p(\partial V^-)=\mu_p(\partial N^+)=\mu_p(\partial N^-)=0.
\end{aligned}
\end{equation}

In the next section, we will count the number of elements in certain subsets of $\Gamma$. Let us collect the definitions here for convenience.
\begin{equation*}
\begin{aligned}
\C(t,\a,\b)&:=\{\c\in \C: N^\a\cap \phi^{-t}\c V^\b\neq \emptyset\},\\
\C_{-\rho}(t,\a,\b)&:=\{\c\in \C: (B_{-\rho}N)^\a\cap \phi^{-t}\c (B_{-\rho}V)^\b\neq \emptyset\},\\
\C^*&:=\{\c\in \C: \c V^+ \subset N^+ \text{\ and\ }\c^{-1}N^-\subset V^-\},\\
\C^*(t,\a,\b)&:=\C^*\cap \C(t,\a,\b).
\end{aligned}
\end{equation*}

\begin{lemma}\label{expansion}
For every $\rho>0$, there exists some $T_2> 0$ such that for all $t\ge T_2$, we have $\C_{-\rho}(t,\a,\b)\subset \C^*(t,\a,\b)$.
\end{lemma}
\begin{proof}
Let $U$ be the interior of $N^+$. Then $(B_{-\rho}N)^+\subset U$. We claim that there exists $T_2> 0$ such that for all $t\ge T_2$ and $\c\in \Gamma$, if $(B_{-\rho}N)^\a\cap \phi^{-t}\c (B_{-\rho}V)^\b\neq \emptyset$, then $\c V^+\subset U$.

Let us prove the claim. Assume not. Then for each $i$, there exist $t_i\to \infty$ and $\c_i\in \C$ such that $v_i\in (B_{-\rho}N)^\a\cap \phi^{-t_i}\c_i (B_{-\rho}V)^\b$, but $\c_iV^+ \nsubseteq U$. Clearly, for any $x\in X$, $\c_ix$ goes to infinity. By passing to a subsequence, let us assume that $\c_ix\to \xi\in \pX$.

By Lemma \ref{diameter}, $(B_{-\rho}N)^\a$ and $(B_{-\rho}V)^\b$ are both compact. By passing to a subsequence, we may assume that $v_i\to v\in (B_{-\rho}N)^\a$ and $\c_i^{-1}\phi^{t_i}v_i\to w\in (B_{-\rho}V)^\b$. Note that $\c_i \pi w\to \xi\in \pX$. Since $d(\c_iw, \phi^{t_i}v_i)\to 0$, we have $\xi=\lim_i \pi \phi^{t_i}v_i\in (B_{-\rho}N)^+$.

We may assume that $\c_i^{-1}\pi v\to \eta\in \pX$. Let $w_i=\c_i^{-1}\phi^{t_i}v_i\in (B_{-\rho}V)^\b$. Then $d(\c_i^{-1}v, \phi^{-t_i}w_i)=d(\c_i^{-1}v, \c_i^{-1}v_i)\to 0$, and thus
$$d(\c_i^{-1}\pi v, \pi \phi^{-t_i}w_i)\to 0.$$
We then see that $\eta=\lim_i \pi \phi^{-t_i}w_i\in (B_{-\rho}V)^-$.

Now we have $\c_i\pi v\to \xi\in (B_{-\rho}N)^+$ and $\c_i^{-1}\pi v\to \eta\in (B_{-\rho}V)^-$. Note that $d_T((B_{-\rho}V)^-,V^+)>\pi$. Applying Theorem \ref{piconvergence} with $\theta=\pi$, we have $\c_iV^+\subset U$ for all $i$ sufficiently large. A contradiction and the claim follows.

By the claim, there exists some $T_2> 0$ such that for all $t\ge T_2$ and $\c\in \C_{-\rho}(t,\a,\b)$, we have $\c V^+ \subset U\subset N^+$.

Analogously, by reversing the roles of $N^\a$ and $V^\b$, and the roles of $\c$ and $\c^{-1}$, we can prove that $\c^{-1} N^- \subset V^-$. Thus $\c\in \C^*$ and the proof of the lemma is completed.
\end{proof}

\section{Using scaling and mixing}
In this section, we prove Theorem \ref{margulis}. First we use the scaling and mixing properties of Knieper measure $m$, to give an asymptotic estimates of $\#\C^*(t,\e^2,\b)$ and $\#\C(t,\e^2,\b)$.

\subsection{Intersection components}
\begin{lemma}\label{nhd}
We have $N\subset N^{\e^2}$, $V\subset V^{\e^2}$.
\end{lemma}
\begin{proof}
Let $w\in N$. By Corollary \ref{equicon}, we have
\begin{equation*}
\begin{aligned}
|s(w)|=|b_{w^-}(\pi w,p)|=|b_{w^-}(\pi w_0,p)|\le |b_{w_0^-}(\pi w_0,p)|+\e^2=\e^2.
\end{aligned}
\end{equation*}
By definition $w\in N^{\e^2}$ and hence $N\subset N^{\e^2}$. $V\subset V^{\e^2}$ can be proved analogously.
\end{proof}

\begin{lemma}\label{intersect1}
For any $t>0$ and $\c\in\C$, we have
$$\{\c\in \C: N^\e\cap \phi^{-t}\c V\neq \emptyset\}\subset \{\c\in \C: N^{\e^2}\cap \phi^{-t}\c V^{\e}\neq \emptyset\}.$$
\end{lemma}
\begin{proof}
Let $\c\in \C$ be such that $N^\e\cap \phi^{-t}\c V\neq \emptyset$. If $v\in \phi^{t}\c^{-1}N^\e\cap  V$, then $w:=\phi^{-t}\c v\in N^\e$. So there exists $w'\in N^{\e^2}$ such that $w=\phi^sw'$ where $|s|\le \e-\e^2$. Then $\phi^{t+s}w'=\phi^t w=\c v$. So $\c^{-1}\phi^{t+s}w'=v\in V\subset V^{\e^2}$ by Lemma \ref{nhd}. We have $\c^{-1}\phi^{t}w'\subset V^{|s|+\e^2}\subset V^\e$.
So $N^{\e^2}\cap \phi^{-t}\c V^{\e}\neq \emptyset$ and the lemma follows.
\end{proof}

\begin{lemma}\label{intersect2}
For any $a>0$, there exists $T_1>0$ large enough such that for any $t\ge T_1$,
\begin{equation*}
\begin{aligned}
&\{\c\in \C: N^\a\cap \phi^{-t}\c V^\b\neq \emptyset\}\\
\subset & \{\c\in\C: (B_aN)^{\a+\b+\e^2}\cap \phi^{-t}\c V\neq \emptyset\}.
\end{aligned}
\end{equation*}
\end{lemma}
\begin{proof}
Let $\c\in \C$ such that $N^\a\cap \phi^{-t}\c V^\b\neq \emptyset$. Then there exists $v\in \phi^{t}N^\a\cap \c V^\b$. So $\c^{-1}v\in V^\b$ and  there exists $w\in V$ such that $w^-=\c^{-1}v^-$. So $|s(w)|\le \e^2$ since $V\subset V^{\e^2}$ by Lemma \ref{nhd}. Moreover, there exists $w'\in W^u(w)$ (hence $s(w)=s(w')$) and $\phi^bw'=\c^{-1}v$ for some $b\in \RR$. It follows that $|b|\le \b+\e^2$ and $d(w,w')\le 4\e$ by Lemma \ref{diameter}. Then $\c w'\in W^u(\c w)$ with $d(\c w,\c w')\le 4\e$, $s(\c w)=s(\c w')$ and $\phi^b \c w'=v$. It follows that $\c w'\in  \phi^tN^{\a+\b+\e^2}$. By comparison theorem (see \eqref{e:comp}), there exists $T_1>0$ large enough such that for any $t\ge T_1$, then $d(\c w,\c w')\le 4\e$ implies that $\c w\in \phi^{[0,\infty]}(B_aN)$. We have that $\phi^{-t}\c w\in (B_aN)^{\a+\b+\e^2}\cap \phi^{-t}\c V$. The lemma follows.
\end{proof}

\subsection{Depth of intersection}

Given $\xi\in \pX$ and $\c\in \C$, define $b_\xi^\c:=b_\xi(\c p, p)$.

\begin{lemma}\label{intersection1}
Let $\xi,\eta\in N^-$, and $\c\in \C(t,\a,\b)$ with $t>0$. Then $|b_\xi^\c-b_\eta^\c|<\e^2$.
\end{lemma}
\begin{proof}
Let $\c\in \C(t,\a,\b)$, so there exists $v\in N^\a\cap \phi^{-t}\c V^\b$. There is $q\in \pi V^\b$ such that $\c q=\pi \phi^t v\in \pi H^{-1}(N^-\times N^+\times [0,\infty))$. Since $p,y\in \F$, we have $d(\c p, \c q)=d(p,q)\le \diam \F+4\e$ by Lemma \ref{diameter}. Thus by Corollary \ref{equicon}, $|b_\xi(\c p, p)-b_\eta(\c p, p)|<\e^2$ for any $\xi,\eta\in N^-$.
\end{proof}

\begin{lemma}\label{intersection3}
Given any $\c\in \C^*(t,\a,\b)$ and any $t\in \RR$, we have
\begin{equation*}
\begin{aligned}
N^{\e^2}\cap \phi^{-t}\c V^\b=\{&w\in E^{-1}(N^-\times \c V^+): \\
&s(w)\in [-\e^2,\e^2]\cap (b_{w^{-}}^\c-t+t_0+[-\b,\b])\}
\end{aligned}
\end{equation*}
where $t_0=s(v_0)$.
\end{lemma}
\begin{proof}
At first, we claim that $N^{\e^2}\cap \phi^{-t}\c V^\b\subset E^{-1}(N^-\times \c V^+)$. Indeed, let $v\in N^{\e^2}\cap \phi^{-t}\c V^\b$. Since $v\in N^{\e^2}$, $v^-\in N^-$. On the other hand, since $v\in \phi^{-t}\c V^\b$, we have $v^+\in \c V^+$. This proves the claim.

Notice that as $\c\in \C^*(t,\a,\b)$, one has $\c V^+\subset N^+$ and $\c^{-1}N^-\subset V^-$. Let $w\in E^{-1}(N^-\times \c V^+)\subset E^{-1}(N^-\times N^+)$. By definition of $N^{\e^2}$, $w\in N^{\e^2}$ if and only if $s(w)\in [-\e^2,\e^2]$.

It remains to show that $w\in \phi^{-t}\c V^\b$ if and only if $s(w)\in (b_{w^{-}}^\c-t+t_0+[-\b,\b])$. To see this, note that
\begin{equation*}
\begin{aligned}
\c V^\b&=\{\c v: v\in E^{-1}(V^-\times V^+) \text{\ and\ }b_{v^-}(\pi v, p)\in t_0+[-\b,\b]\}\\
&=\{w\in E^{-1}(\c V^-\times \c V^+): b_{w^-}(\pi w, \c p)\in t_0+[-\b,\b]\}.
\end{aligned}
\end{equation*}
Since $s(\phi^t w)=s(w)+t$ and
$$b_{w^-}(\pi w, \c p)=b_{w^-}(\pi w, p)+b_{w^-}(p, \c p)=s(w)-b_{w^{-}}^\c,$$
we know $\phi^t w\in \c V^\b$ if and only if $s(w)-b_{w^{-}}^\c+t\in t_0+[-\b,\b]$. The lemma follows.
\end{proof}

The following lemma implies that the intersection components also have product structure.
\begin{lemma}\label{depth1}
If $\c\in \C^*(t,\e^2,\b)$, then
$$N^{\e^2}\cap \phi^{-(t+4\e^2)}\c V^{\b+8\e^2}\supset H^{-1}(N^-\times \c V^+\times [-\e^2,\e^2]):=N^\c.$$
\end{lemma}
\begin{proof}
Let $\c\in \C^*(t,\e^2,\b)$, then $N^{\e^2}\cap \phi^{-t}\c V^{\b}\neq \emptyset$. By Lemma \ref{intersection3}, there exists $\eta\in N^-$ such that
$$[-\e^2,\e^2]\cap (b_{\eta}^\c-t+t_0+[-\b,\b])\neq \emptyset.$$
It follows that $[-\e^2,\e^2]\subset (b_{\eta}^\c-t-2\e^2+t_0+[-\b,\b+4\e^2]).$ Then by Lemma \ref{intersection1}, for any $\xi\in N^-$ we have
$$[-\e^2,\e^2]\cap (b_{\xi}^\c-t-2\e^2+t_0+[-\b,\b+4\e^2])\neq \emptyset,$$
which in turn implies that
$$[-\e^2,\e^2]\subset  (b_{\xi}^\c-t-4\e^2+t_0+[-\b,\b+8\e^2]).$$
We are done by Lemma \ref{intersection3}.
\end{proof}
\subsection{Scaling and mixing calculation}
We use the following notations in the asymptotic estimates.
\begin{equation*}
\begin{aligned}
f(t)=e^{\pm C}g(t)&\Leftrightarrow e^{-C}g(t)\le f(t)\le e^{C}g(t) \text{\ for all\ } t;\\
f(t) \lesssim g(t) &\Leftrightarrow \limsup_{t\to \infty}\frac{f(t)}{g(t)}\le 1;\\
f(t) \gtrsim g(t) &\Leftrightarrow \liminf_{t\to \infty}\frac{f(t)}{g(t)}\ge 1;\\
f(t) \sim g(t) &\Leftrightarrow \lim_{t\to \infty}\frac{f(t)}{g(t)}= 1;\\
f(t)\sim e^{\pm C}g(t)&\Leftrightarrow e^{-C}g(t)\lesssim f(t)\lesssim e^{C}g(t).
\end{aligned}
\end{equation*}

\begin{lemma}\label{tscaling1}
If $\c\in \C^*(t,\e^2,\b)$, then
$$\frac{m(N^\c)}{m(V^\b)}=e^{\pm 26h\e}e^{ht_0}e^{-ht}\frac{\e^2\mu_p(N^-)}{\b\mu_p(V^-)}$$
where $N^\c$ is from Lemma \ref{depth1}.
\end{lemma}

\begin{proof}
The main work is to estimate $\b_p(\xi,\c\eta)$ and $b_\eta(\c^{-1}p,p)$ for any $\xi\in N^-$ and $\eta\in V^+$.

Firstly, take $q$ lying on the geodesic connecting $\xi$ and $\c\eta$ such that $b_\xi(q,p)=0$. Then
\begin{equation}\label{e:smallBuse}
\begin{aligned}
|\b_p(\xi,\c\eta)|=|b_\xi(q,p)+b_{\c\eta}(q,p)|=|b_{\c\eta}(q,p)|\le d(q,p)<4\e
\end{aligned}
\end{equation}
where we used Lemma \ref{diameter} in the last inequality.

Secondly, since $\c\in \C^*(t,\e^2,\b)$, there exist $v\in V^\b$ and $w\in N^{\e^2}$ such that $w=\phi^{-t}\c v$. Take $q'$ lying on the geodesic connecting $\c^{-1}\xi$ and $\eta$ such that $b_{\c^{-1}\xi}(q',p)=t_0$. Then $q'\in V^\b$. Define $q''$ to be the unique point in $\pi H^{-1}(N^-\times N^+\times \{0\})\cap c_{\xi,\c\eta}$. Then using Lemma \ref{diameter},
\begin{equation*}
\begin{aligned}
d(q'',\c q')&\le d(q'',\pi w)+d(\pi w,\pi \phi^t w)+d(\c \pi v, \c q')\le t+8\e\\
d(q'',\c q')&\ge d(\pi w,\pi \phi^t w)-d(q'',\pi w)-d(\c \pi v, \c q')\ge t-8\e.
\end{aligned}
\end{equation*}
Noticing that $q'',\c q'$ lie on the geodesic $c_{\xi, \c \eta}$, we have
\begin{equation*}
\begin{aligned}
&b_\eta(\c^{-1}p,p)=b_{\c\eta}(p,\c p)=b_{\c\eta}(p,\c q')+b_{\c\eta}(\c q',\c p)\\
\le &b_{\c\eta}(q'',\c q')+d(q'',p)+b_{\eta}(q',p)\le t+12\e+b_{\eta}(q',p)
\end{aligned}
\end{equation*}
and
\begin{equation*}
\begin{aligned}
&b_\eta(\c^{-1}p,p)=b_{\c\eta}(p,\c p)= b_{\c\eta}(p,\c q')+b_{\c\eta}(\c q',\c p)\\
\ge &b_{\c\eta}(q'',\c q')-d(q'',p)+b_{\eta}(q', p)\ge t-12\e+b_{\eta}(q',p).
\end{aligned}
\end{equation*}
Thus we have
\begin{equation*}
\begin{aligned}
\frac{m(N^\c)}{m(V^\b)}&=\frac{2\e^2}{2\b}\frac{\int_{N^-}\int_{V^+}e^{h\b_p(\xi,\c\eta)}d\mu_p(\xi)d\mu_{\c^{-1}p}(\eta)}
{\int_{V^-}\int_{V^+}e^{h\b_p(\xi',\eta')}d\mu_p(\xi')d\mu_{p}(\eta')}\\
&=\frac{\e^2}{\b}e^{\pm 4h\e}\frac{\int_{N^-}\int_{V^+}e^{-h b_\eta(\c^{-1}p,p)}d\mu_p(\xi)d\mu_{p}(\eta)}
{\int_{V^-}\int_{V^+}e^{h\b_p(\xi',\eta')}d\mu_p(\xi')d\mu_{p}(\eta')}\\
&=\frac{\e^2}{\b} e^{\pm 16h\e}e^{-ht}\frac{\int_{N^-}\int_{V^+}e^{-hb_\eta(q',p)}d\mu_p(\xi)d\mu_p(\eta)}{\int_{V^-}\int_{V^+}e^{\pm 8h\e}e^{-h(b_{\xi'}(q',p)+b_{\eta'}(q',p))}
d\mu_p(\xi')d\mu_{p}(\eta')}\\
&=\frac{\e^2}{\b} e^{\pm 24h\e}e^{-ht}e^{\pm 2h\e^2}e^{ht_0}\frac{\mu_p(N^-)}{\mu_p(V^-)}\\
&=e^{\pm 26h\e}e^{ht_0}e^{-ht}\frac{\e^2\mu_p(N^-)}{\b\mu_p(V^-)}.
\end{aligned}
\end{equation*}
where in the third equality we used the fact that $c_{\xi'\eta'}$ passes through a point in $V^\b$, within a distance $4\e$ from $q'$, and in the fourth equality we used Corollary \ref{equicon} and $b_{\c^{-1}\xi}(q',p)=t_0$.
\end{proof}

Finally, we combine scaling and mixing properties of Knieper measure to obtain the following asymptotic estimates.
\begin{proposition}\label{asymptotic}
We have
\begin{equation*}
\begin{aligned}
e^{-30h\e} \lesssim &\frac{\#\C^*(t,\e^2,\b)e^{ht_0}}{\mu_p(V^-)\mu_p(N^+)e^{ht}}\frac{1}{2\b}
\lesssim e^{30h\e}(1+\frac{8\e^2}{\b}),\\
e^{-30h\e} \lesssim &\frac{\#\C(t,\e^2,\b)e^{ht_0}}{\mu_p(V^-)\mu_p(N^+)e^{ht}}\frac{1}{2\b}
\lesssim e^{30h\e}(1+\frac{8\e^2}{\b}).
\end{aligned}
\end{equation*}
\end{proposition}
\begin{proof}
By Lemmas \ref{expansion} and \ref{depth1}, for any $0<\rho<\theta$ and $t$ large enough, we have
$$(B_{-\rho}\lN)^{\e^2}\cap \phi^{-t}(B_{-\rho}\lV)^{\b}\subset \bigcup_{\c\in \C^*(t,\e^2,\b)}\lN^\c \subset \lN^{\e^2} \cap \phi^{-(t+4\e^2)}\lV^{\b+8\e^2}.$$
By Lemma \ref{tscaling1}, $$\frac{m(N^\c)}{m(V^\b)}=e^{\pm 26h\e}e^{ht_0}e^{-ht}\frac{\e^2\mu_p(N^-)}{\b\mu_p(V^-)}.$$
Estimating similarly to \eqref{e:smallBuse},
\begin{equation}\label{e:N}
\begin{aligned}
m(N^{\e^2})=2\e^2\int_{N^-}\int_{N^+}e^{h\b_p(\xi,\eta)}d\mu_p(\xi)d\mu_{p}(\eta)=2\e^2e^{\pm 4h\e}\mu_p(N^-)\mu_p(N^+).
\end{aligned}
\end{equation}
Thus we have
\begin{equation*}
\begin{aligned}
e^{-26h\e}\lm((B_{-\rho}\lN)^{\e^2}\cap \phi^{-t}(B_{-\rho}\lV)^{\b})&\le \#\C^*(t,\e^2,\b)e^{ht_0}e^{-ht}\frac{\e^2\mu_p(N^-)}{\b\mu_p(V^-)}\lm(\lV^\b)\\
&\le e^{26h\e}\lm(\lN^{\e^2} \cap \phi^{-(t+4\e^2)}\lV^{\b+8\e^2}).
\end{aligned}
\end{equation*}
Dividing by $\lm(\lN^{\e^2})\lm(\lV^\b)$ and using mixing of $\lm$, we get
\begin{equation*}
\begin{aligned}
e^{-26h\e}\frac{m((B_{-\rho}N)^{\e^2})m((B_{-\rho}V)^{\b})}{m(N^{\e^2})m(V^\b)}& \lesssim \frac{\#\C^*(t,\e^2,\b)e^{ht_0}}{e^{ht}m(N^{\e^2})}\frac{\e^2\mu_p(N^-)}{\b\mu_p(V^-)}\\
&\lesssim e^{26h\e}\frac{m(V^{\b+8\e^2})}{m(V^\b)}.
\end{aligned}
\end{equation*}
By \eqref{e:boundary}, letting $\rho \to 0$, we obtain
\begin{equation*}
\begin{aligned}
e^{-26h\e} \lesssim \frac{\#\C^*(t,\e^2,\b)e^{ht_0}}{e^{ht}m(N^{\e^2})}\frac{\e^2\mu_p(N^-)}{\b\mu_p(V^-)}
\lesssim e^{26h\e}(1+\frac{8\e^2}{\b}).
\end{aligned}
\end{equation*}
Thus by \eqref{e:N}
\begin{equation}\label{e:sim1}
\begin{aligned}
e^{-30h\e} \lesssim \frac{\#\C^*(t,\e^2,\b)e^{ht_0}}{\mu_p(V^-)\mu_p(N^+)e^{ht}}\frac{1}{2\b}
\lesssim e^{30h\e}(1+\frac{8\e^2}{\b}).
\end{aligned}
\end{equation}

To prove the second equation, we consider $\rho>0$. Then by Lemma \ref{expansion}, $\C^*(t,\e^2,\b)\subset \C(t,\e^2,\b) \subset \C^*_\rho(t,\e^2,\b).$ By \eqref{e:sim1},
\begin{equation*}
\begin{aligned}
e^{-30h\e}& \lesssim \frac{\#\C^*(t,\e^2,\b)e^{ht_0}}{\mu_p(V^-)\mu_p(N^+)e^{ht}}\frac{1}{2\b}\le \frac{\#\C(t,\e^2,\b)e^{ht_0}}{\mu_p(V^-)\mu_p(N^+)e^{ht}}\frac{1}{2\b}\\
&\le \frac{\#\C^*_\rho(t,\e^2,\b)e^{ht_0}}{\mu_p(V^-)\mu_p(N^+)e^{ht}}\frac{1}{2\b}\\
&\lesssim e^{30h\e}(1+\frac{8\e^2}{\b})\frac{\mu_p((B_\rho V)^-)\mu_p((B_\rho N)^+)}{\mu_p(V^-)\mu_p(N^+)}.
\end{aligned}
\end{equation*}
Letting $\rho\searrow 0$ and by \eqref{e:boundary}, we get the second equation in the proposition.
\end{proof}

\subsection{Integration}
Let $V\subset S_yX, N\subset S_xX$ be as above, and $0\le a<b$. Let $n(a,b,V,N^0)$ denote the number of connected components at which $\phi^{[-b,-a]}\lV$ intersects $\lN^0$. $n_t(V^\b,N^\a)$ (resp. $n_t(V,N^\a)$) denotes the number of connected components at which $\phi^{-t}\lV^\b$  (resp. $\phi^{-t}\lV$) intersects $\lN^\a$.

\begin{lemma}\label{asy1}
We have
$$n_t(V,N^\e)\lesssim e^{-ht_0}\mu_p(V^-)\mu_p(N^+)e^{ht} 2\e (1+8\e)e^{30 h\e}.$$
\end{lemma}
\begin{proof}
Setting $\a=\e^2$ in Lemma \ref{intersect1},
$$n_t(V,N^{\e})\le n_t(V^\e,N^{\e^2}).$$
Setting $\b=\e$ in Proposition \ref{asymptotic},
\begin{equation*}
\begin{aligned}
&n_t(V^\e,N^{\e^2})=\#\C(t,\e^2,\e)
\lesssim  &e^{-ht_0}\mu_p(V^-)\mu_p(N^+)e^{ht} 2\e (1+8\e)e^{30 h\e}.
\end{aligned}
\end{equation*}
\end{proof}

\begin{lemma}\label{asy2}
We have
$$n_t(V,N^\e)\gtrsim e^{-ht_0}\mu_p(V^-)\mu_p(N^+)e^{ht} 2\e (1-2\e)e^{-30 h\e}.$$
\end{lemma}
\begin{proof}
Setting $\a=\e^2$, $\b=\e-2\e^2$ in Lemma \ref{intersect2}, for any $a>0$, we have
$$n_t(V,N^{\e})\ge n_t(V^{\e-2\e^2},(B_{-a}N)^{\e^2})$$
for any $t\ge T_1$ where $T_1$ is provided by Lemma \ref{intersect2}.
Setting $\b=\e-2\e^2$ in Lemma \ref{asymptotic},
\begin{equation*}
\begin{aligned}
n_t(V^{\e-2\e^2},(B_{-a}N)^{\e^2})\gtrsim  &e^{-ht_0}\mu_p(V^-)\mu_p((B_{-a}N)^+)e^{ht}2\e(1-2\e) e^{-30 h\e}.
\end{aligned}
\end{equation*}
Letting $a\to 0$, by \eqref{e:boundary} we obtain the conclusion of the lemma.
\end{proof}

\begin{proposition}\label{key}
There exists $Q>0$ such that
\begin{equation*}
\begin{aligned}
e^{-2Q\e}e^{-ht_0}\mu_p(V^-)\mu_p(N^+)\frac{1}{h}e^{ht} &\lesssim n(0,t,V,N^0)\\
&\lesssim  e^{2Q\e}e^{-ht_0}\mu_p(V^-)\mu_p(N^+)\frac{1}{h}e^{ht}.
\end{aligned}
\end{equation*}
\end{proposition}
\begin{proof}
It is clear that for any $b>0$,
\begin{equation}\label{e:integral}
\begin{aligned}
n(0,t,V,N^0)\thicksim n(b,t,V,N^0).
\end{aligned}
\end{equation}
By Lemmas \ref{asy1} and \ref{asy2}, we can choose $b$ large enough such that
$$n(t-\e,t+\e,V,N^0)=n_t(V,N^\e)= e^{\pm Q\e}e^{-ht_0}\mu_p(V^-)\mu_p(N^+)e^{ht} 2\e$$
for some $Q>2h$ large enough and for all $t\ge b$.
Let $t_k=b+\e+2k\e$, then
\begin{equation*}
\begin{aligned}
n(b, t,V,N^0)&\le \sum_{k=0}^{[\frac{t-b}{2\e}]+1}n(t_k-\e,t_k+\e,V,N^0)\\
&\le  e^{Q\e}e^{-ht_0}\mu_p(V^-)\mu_p(N^+)\sum_{k=0}^{[\frac{t-b}{2\e}]+1}2\e e^{ht_k} \\
&\le e^{Q\e}e^{-ht_0}\mu_p(V^-)\mu_p(N^+)\int_{b-\e}^{t+2\e}e^{hs}ds\\
&=  e^{Q\e}e^{-ht_0}\mu_p(V^-)\mu_p(N^+)\frac{1}{h}(e^{h(t+2\e)}-e^{h(b-\e)})\\
&\le e^{2Q\e}e^{-ht_0}\mu_p(V^-)\mu_p(N^+)\frac{1}{h}e^{ht}
\end{aligned}
\end{equation*}
and
\begin{equation*}
\begin{aligned}
n(b, t,V,N^0)&\ge \sum_{k=0}^{[\frac{t-b}{2\e}]-1}n(t_k-\e,t_k+\e,V,N^0)\\
&\ge e^{- Q\e}e^{-ht_0}\mu_p(V^-)\mu_p(N^+)\sum_{k=0}^{[\frac{t-b}{2\e}]-1}2\e e^{ht_k} \\
&\ge e^{-Q\e}e^{-ht_0}\mu_p(V^-)\mu_p(N^+)\int_{b+\e}^{t-2\e}e^{hs}ds\\
&\ge  e^{-Q\e}e^{-ht_0}\mu_p(V^-)\mu_p(N^+)\frac{1}{h}(e^{h(t-2\e)}-e^{h(b+\e)})\\
&\gtrsim e^{-2Q\e}e^{-ht_0}\mu_p(V^-)\mu_p(N^+)\frac{1}{h}e^{ht}.
\end{aligned}
\end{equation*}
The lemma then follows from \eqref{e:integral}.
\end{proof}

\subsection{Summing over the regular partition-cover}
Denote $$a_t(x,y):=\#\{\c\in \C: \c y\in B(x,t)\}$$
and $$a^{1}_t(x,y):=\#\{\c\in \C: \c y\in B(x,t), \text{\ and\ }c_{x,\c y} \text{\ is singular}\}.$$
It is easy to see that $b_t(x)=\int_\F a_{t}(x,y)d  \text{Vol} (y)$. In the following, we give an asymptotic estimates of $a_t(x,y)$.
\begin{lemma}\label{singular}
There exist $C>0$ and $0<h'<h$, such that for any $x,y\in \F$,
$$a^{1}_t(x,y) \le Ce^{h't}.$$
\end{lemma}
\begin{proof}
Given any $\e<\inj M/5$ and $t>0$, let $\c_1\neq \c_2\in \C$ be such that $\c_1y,\c_2 y\in B(x,t)\setminus B(x,t-\e)$, and both $c_{x,\c_1 y}$ and $c_{x,\c_2 y}$ are singular. Then it is easy to see that $\dot c_{x,\c_1 y}$ and $\dot c_{x,\c_2 y}$ are $(t,\e)$-separated. By a result of Knieper \cite[Theorem 1.1]{Kn2}, the topological entropy of the singular set $h_\top(\Sing)$ is strictly smaller than $h$. It follows that the number of $\c\in \C$ as above is less than $C_1e^{h't}$ for some $C_1>0$ and $h_\top(\Sing)< h'<h$.

Let $t_i=i \e$, then
$$a^{1}_t(x,y) \le \sum_{i=1}^{[t/\e]+1} C'e^{h't_i}= \frac{C'}{\e}\sum_{i=1}^{[t/\e]+1}\e e^{h't_i}\le \frac{C'}{\e}\int_{0}^{t+\e}e^{h's}ds\le Ce^{h't}$$
for some $C>0$.
\end{proof}

Denote by
$$a(0,t, x,y, \cup_{j=n+1}^\infty N_j):=\#\{\c\in \C: \c y\in \pi\phi^{[0,t]}\cup_{j=n+1}^\infty N_j\},$$
and similarly,
$$a(0,t, x,y, \cup_{i=m+1}^\infty V_i):=\#\{\c\in \C: \c^{-1} x\in \pi\phi^{[-t,0]}\cup_{i=m+1}^\infty V_i\}.$$

\begin{lemma}\label{nonuniform}
There exists $C>0$ such that
$$\limsup_{t\to \infty}e^{-ht}a(0,t, x,y, \cup_{j=n+1}^\infty N_j)\le C\cdot \mu_p((\cup_{j=n+1}^\infty N_j)^+),$$
and
$$\limsup_{t\to \infty}e^{-ht}a(0,t, x,y, \cup_{i=m+1}^\infty V_i)\le C\cdot \mu_p((\cup_{i=m+1}^\infty V_i)^-).$$
\end{lemma}
\begin{proof}
We start to estimate from above the spherical volume of $\pi\phi^t\cup_{j=n+1}^\infty N_j$ which is a subset of the sphere $S(x,t)$ of radius $t$ around $x$.

Let $x_1,\cdots, x_k$ be a maximal $\rho$-separated subset of $\pi\phi^t\cup_{j=n+1}^\infty N_j$, where $\rho\ge R$ from Shadowing Lemma \ref{shadow}. Then $B(x_i,\rho/2), i=1,2,\cdots, k$ are disjoint. By comparison theorem (see \eqref{e:comp}), for any $a>0$ there exists $T_3>0$ such that if $t\ge T_3$, then $f_{x}B(x_i,\rho/2)\subset (B_{a}\cup_{j=n+1}^\infty N_j)^+$. By Shadowing Lemma \ref{shadow} and the fact that $p=x$, we know $\mu_p(f_{x}B(x_i,\rho/2))\ge b^{-1}e^{-ht}$ and hence
$$k\le be^{ht}\mu_p((B_{a}\cup_{j=n+1}^\infty N_j)^+).$$
From the uniformity of the geometry, there exists $l>0$ such that $\text{Vol} (B(x_i,\rho)\cap S(x,t))\le l$ for each $1\le i\le k$.
So
$$ \text{Vol} \pi\phi^t\cup_{j=n+1}^\infty N_j\le lk\le lbe^{ht}\mu_p((B_{a}\cup_{j=n+1}^\infty N_j)^+).$$
Then there exists $C_1, C_2>0$ such that
\begin{equation}\label{e:volume}
\begin{aligned}
 \text{Vol} \pi\phi^{[0,t]}\cup_{j=n+1}^\infty N_j\le &C_1+\int_{T_3}^t  \text{Vol} \phi^s\cup_{j=n+1}^\infty N_j ds\\
\le &C_1+C_2e^{ht}\mu_p((B_{a}\cup_{j=n+1}^\infty N_j)^+).
\end{aligned}
\end{equation}
Note that $C_2$ is independent of $T_3$ and $a$.

Now since each $\c\F$ has equal finite diameter and volume, there exists $T_4>0$ such that
\begin{equation}\label{e:up}
\begin{aligned}
a(0,t, x,y, \cup_{j=n+1}^\infty N_j)&\le C_3+C_4  \text{Vol} \pi\phi^{[T_4,t]}B_{a}\cup_{j=n+1}^\infty N_j\\
&\le C_5+C_6e^{ht}\mu_p((B_{2a}\cup_{j=n+1}^\infty N_j)^+)
\end{aligned}
\end{equation}
where we used \eqref{e:volume} in the last inequality. Note that $C_6$ is independent of $a$.
Thus
$$\limsup_{t\to \infty}e^{-ht}a(0,t, x,y, \cup_{j=n+1}^\infty N_j)\le C_6\cdot \mu_p((B_{2a}\cup_{j=n+1}^\infty N_j)^+).$$
As $a>0$ could be arbitrarily small and $C_6$ is independent of $a$, we get the first inequality in the lemma.

The second inequality can be proved analogously with minor modification: When applying Shadowing Lemma \ref{shadow}, we transfer from $\mu_p$ to $\mu_y$ by
$$\frac{d\mu_y}{d\mu_p}(\xi)=e^{-hb_\xi(y,p)}\le e^{hd(y,p)}\le e^{h\diam \F}$$
for any $\xi\in \pX$.
\end{proof}

\begin{proof}[Proof of Theorem \ref{margulis}]
Since the diameter of each flow box is no more than $4\e$, we have
\begin{equation}\label{e:less1}
\begin{aligned}
n(0,t, \cup V_i, \cup (N_j)^0)\le a_{t+4\e}(x,y)
\end{aligned}
\end{equation}
and
\begin{equation}\label{e:less2}
\begin{aligned}
a_{t-4\e}(x,y)&\le a^{1}_{t-4\e}(x,y)+ a(0,t, x,y, \cup_{j=n+1}^\infty N_j)\\
&+a(0,t, x,y, \cup_{i=m+1}^\infty V_i)+ n(0,t, \cup_{i=1}^m V_i, \cup_{j=1}^n (N_j)^0).
\end{aligned}
\end{equation}

For each $V_i$, denote by $t_0^i:=b_{v_i^-}(\pi v_i,p)$ where $v_i\in V_i$. Recall that in subsection 3.1 we suppressed $i$ and write $t_0=t_0^i$, since only one $V_i$ is considered there. By Proposition \ref{key}, for any $m,n\in \NN$
\begin{equation*}
\begin{aligned}
&\liminf_{t\to \infty}e^{-ht}n(0,t, \cup V_i, \cup (N_j)^0)\\
\ge&\liminf_{t\to \infty}e^{-ht}\sum_{i=1}^m\sum_{j=1}^n n(0,t, V_i,  (N_j)^0)\\
\ge &\sum_{i=1}^m\sum_{j=1}^n e^{-2Q\e}e^{-hb_{v_i^-}(\pi v_i,p)}\mu_p((V_i)^-)\mu_p((N_j)^+)\frac{1}{h}.
\end{aligned}
\end{equation*}
Note that $v\mapsto b_{v^-}(\pi v,p)$ is a continuous function by Lemma \ref{continuous}. So if we choose a sequence of finer and finer regular partition-covers, and let $m,n\to \infty$ on the right hand,
\begin{equation*}
\begin{aligned}
&\liminf_{t\to \infty}e^{-ht}n(0,t, \cup V_i, \cup (N_j)^0)\\
\ge &e^{-2Q\e}\frac{1}{h}\int_{S_xM\cap \Reg}\int_{S_y M\cap \Reg}e^{-hb_{v^-}(\pi v,p)}d\tilde\mu_y(-v)d\tilde\mu_x(w).
\end{aligned}
\end{equation*}
Then by \eqref{e:less1} we have
\begin{equation*}
\begin{aligned}
a_{t+4\e}(x,y)&\gtrsim e^{-2Q\e}\frac{1}{h}e^{ht} \int_{S_xM\cap \Reg}\int_{S_y M\cap \Reg}e^{-hb_{v^-}(\pi v,p)}d\tilde\mu_y(-v)d\tilde\mu_x(w).
\end{aligned}
\end{equation*}
Replacing $t$ by $t-\e$, we have
\begin{equation}\label{e:ge}
\begin{aligned}
a_{t}(x,y)&\gtrsim e^{-2Q\e}e^{-4\e}\frac{1}{h}e^{ht} \int_{S_xM\cap \Reg}\int_{S_y M\cap \Reg}e^{-hb_{v^-}(\pi v,p)}d\tilde\mu_y(-v)d\tilde\mu_x(w).
\end{aligned}
\end{equation}

On the other hand, by Proposition \ref{key} and Corollary \ref{equicon}, for any $m,n\in \NN$
\begin{equation}\label{e:upper}
\begin{aligned}
&e^{2Q\e}\frac{1}{h} \int_{S_xM\cap \Reg}\int_{S_y M\cap \Reg}e^{-hb_{v^-}(\pi v,p)}d\tilde\mu_y(-v)d\tilde\mu_x(w)\\
\ge &\sum_{i=1}^m\sum_{j=1}^n e^{2Q\e}e^{-h(t_0^i+\e^2+4\e)}\mu_p((V_i)^-)\mu_p((N_j)^+)\frac{1}{h}\\
\ge &\limsup_{t\to \infty}\sum_{i=1}^m\sum_{j=1}^n n(0,t, V_i,  (N_j)^0)e^{-ht}e^{-h(\e^2+4\e)}.
\end{aligned}
\end{equation}
\begin{comment}
Letting $m,n\to \infty$ on the right hand side, we have
\begin{equation}\label{e:upper}
\begin{aligned}
&e^{2Q\e}\frac{1}{h}\int_{S_xM}\int_{S_y M}e^{-hb_{v^-}(\pi v,p)}d\tilde\mu_p(-v)d\tilde\mu_p(w)\\
\ge &\limsup_{t\to \infty}e^{-ht} n(0,t, \cup V_i, \cup (N_j)^0).
\end{aligned}
\end{equation}
\end{comment}
Combining with \eqref{e:less2} and Lemmas \ref{singular}, \ref{nonuniform}, one has
\begin{equation*}
\begin{aligned}
&a_{t-4\e}(x,y)\lesssim  Ce^{h'(t-4\e)}\\+&Ce^{h(t-4\e)}\cdot \mu_p((\cup_{j=n+1}^\infty N_j)^+)+ Ce^{h(t-4\e)}\cdot \mu_p((\cup_{j=m+1}^\infty V_i)^+)\\
+&e^{2Q\e}e^{h(\e^2+4\e)}\frac{1}{h}e^{ht} \int_{S_xM\cap \Reg}\int_{S_y M\cap \Reg}e^{-hb_{v^-}(\pi v,p)}d\tilde\mu_y(-v)d\tilde\mu_x(w).
\end{aligned}
\end{equation*}
Letting $m,n\to \infty$ and replacing $t-4\e$ by $t$, we have
\begin{equation}\label{e:le}
\begin{aligned}
&a_{t}(x,y)\\
\lesssim &e^{2Q\e}e^{h(\e^2+4\e)+4\e}\frac{1}{h}e^{ht} \int_{S_xM\cap \Reg}\int_{S_y M\cap \Reg}e^{-hb_{v^-}(\pi v,p)}d\tilde\mu_y(-v)d\tilde\mu_x(w).
\end{aligned}
\end{equation}

Letting $\e\to 0$ in \eqref{e:ge} and \eqref{e:le} and recalling that $p=x$, we get
$$a_{t}(x,y)\sim \frac{1}{h}e^{ht}\cdot c(x,y)$$
where
$$c(x,y):=\int_{S_xX\cap \Reg}\int_{S_y X\cap \Reg }e^{-hb_{v^-}(\pi v,x)}d\tilde\mu_y(-v)d\tilde\mu_x(w).$$
By Lemma \ref{null}, in fact we have
$$c(x,y)=\int_{S_xX}\int_{S_y X}e^{-hb_{v^-}(\pi v,x)}d\tilde\mu_y(-v)d\tilde\mu_x(w).$$

It follows that
$b_t(x)=\int_\F a_{t}(x,y)d  \text{Vol} (y)\sim  \frac{1}{h}e^{ht}\int_\F c(x,y)d  \text{Vol}(y)$ by the dominated convergence Theorem. Indeed, by \eqref{e:less2}, Lemma \ref{singular}, \eqref{e:up} and \eqref{e:upper},
$$e^{-ht}a_{t}(x,y)\le B_1+B_2c(x,y)$$
where the right hand side is integrable.

Define $c(x):=\int_\F c(x,y)d  \text{Vol}(y)$, we get $b_t(x)\sim c(x)\frac{e^{ht}}{h}$.
Obviously, $c(\c x)=c(x)$ for any $\c\in \C$. So $c$ descends to a function from $M$ to $\RR$, which still denoted by $c$.

It remains to prove the continuity of $c$. Let $y_n\to x\in X$. For any $\e>0$, there exists $n\in \NN$ such that $d(x,y_n)<\e$. Then
$$b_{t-\e}(x)\le b_{t}(y_n)\le b_{t+\e}(x).$$
We have
\begin{equation*}
\begin{aligned}
c(y_n)=\lim_{t\to \infty}\frac{b_t(y_n)}{e^{ht}/h}\le \lim_{t\to \infty}\frac{b_{t+\e}(x)}{e^{ht}/h}=c(x)e^{h\e}
\end{aligned}
\end{equation*}
and
\begin{equation*}
\begin{aligned}
c(y_n)=\lim_{t\to \infty}\frac{b_t(y_n)}{e^{ht}/h}\ge \lim_{t\to \infty}\frac{b_{t-\e}(x)}{e^{ht}/h}=c(x)e^{-h\e}.
\end{aligned}
\end{equation*}
Thus $\lim_{n\to \infty}c(y_n)=c(x)$ and $c$ is continuous.
\end{proof}

\section{Properties of the Margulis function}
We prove Theorem \ref{function} in this section. Let $M$ be a rank one closed $C^\infty$ Riemannian manifold of nonpositive curvature, $X$ the universal cover of $M$. Firstly, we give an equivalent definition for the Patterson-Sullivan measure via the Margulis function.

Recall that $f_x: S_xX\to \partial X$, $f_x(v)=v^+$. Similarly, we define the canonical projection $f_x^R: S(x,R)\to \partial X$ by $f_x^R(y)=v_y^+$ where $v_y$ is the unit normal vector of the sphere $S(x,R)$ at $y$. For any continuous function $\varphi: \partial X \to \RR$, define a measure on $\pX$ by
$$\nu_x^R(\varphi):=\frac{1}{e^{hR}}\int_{S(x,R)}\varphi \circ f_x^R(y)d \text{Vol} (y).$$
By Theorem \ref{margulis} or \eqref{e:sphere}, $\nu_x^R(\partial X)$ is uniformly bounded from above and below, and hence there exist limit measures of $\nu_x^R$ when $R\to \infty$. Take any limit measure $\nu_x$. By Theorem \ref{margulis} or \eqref{e:sphere}, we see that
$$\nu_x(\partial X)=\lim_{R\to \infty}\frac{s_R(x)}{e^{hR}}=c(x).$$
Moreover, by definition one can check that
\begin{enumerate}
\item For any $p,\ q \in X$ and $\nu_{p}$-a.e. $\xi\in \pX$,
$$\frac{d\nu_{q}}{d \nu_{p}}(\xi)=e^{-h \cdot b_{\xi}(q,p)}.$$
% where $b_{\xi}(q,p)$ is the Busemann function.
\item $\{\nu_{p}\}_{p\in X}$ is $\Gamma$-equivariant, i.e., for every Borel set $A \subset \pX$ and for any $\c \in \Gamma$,
we have $$\nu_{\c p}(\c A) = \nu_{p}(A).$$
\end{enumerate}
(2) is obvious. Let us prove (1). Let $\rho>0$ be arbitrarily small. Take a small compact neighborhood $U_\xi$ of $\xi$ in $\pX$ such that $|b_{\xi}(q,p)-b_{\xi'}(q,p)|<\rho$ for any $\xi'\in U_\xi$. Let $R>0$ be large enough such that
\begin{enumerate}
  \item $|\frac{\text{Vol}(A_R)}{e^{hR}}-\nu_p(U_{\xi})|<\rho$ where $A_R=(f_p^R)^{-1}U_\xi$;
  \item $|d(q,c_{p,\xi'}(R))-R-b_{\xi'}(q,p)|<\rho$ for any $\xi'\in U_\xi$.
\end{enumerate}
Now we divide $U_\xi$ into finitely many sufficiently small compact subsets $U_\xi^i\subset U_\xi, i=1,\cdots,k$ such that the following holds. By enlarging $R$ if necessary,
$$\left|\frac{\text{Vol}(\bar A_R^i)}{\text{Vol}( A_R^i)}-1\right|<\rho$$
where $A_R^i=(f_p^R)^{-1}U_\xi^i$, $\bar A_R^i=(f_q^{d(q,c_{p,\xi_i}(R))})^{-1}U_\xi^i$ and $\xi_i\in U_\xi^i$.
Then
\begin{equation*}
\begin{aligned}
\nu_q(U_\xi)=&\sum_{i=1}^k\nu_q(U_\xi^i)\le \sum_{i=1}^k\nu_q^{d(q,c_{p,\xi_i}(R))}(U_\xi^i)+\rho\\
\le&\frac{1}{e^{h(R+b_{\xi_i}(q,p)-\rho)}}\sum_{i=1}^k \text{Vol}(\bar A_R^i)+\rho\\
\le &\frac{1+\rho}{e^{h(R+b_{\xi}(q,p)-2\rho)}}\sum_{i=1}^k \text{Vol}(A_R^i)+\rho\\
\le &\frac{1+\rho}{e^{h(b_{\xi}(q,p)-2\rho)}}\nu_p^R (U_\xi)+\rho\\
\le &\frac{1+\rho}{e^{h(b_{\xi}(q,p)-2\rho)}}(\nu_p (U_\xi)+\rho)+\rho.
\end{aligned}
\end{equation*}
As $U_\xi$ shrinks to $\{\xi\}$, $\rho>0$ could be arbitrarily small. So $\frac{d\nu_{q}}{d \nu_{p}}(\xi)\le e^{-h \cdot b_{\xi}(q,p)}$. By symmetry, we get $\frac{d\nu_{q}}{d \nu_{p}}(\xi)= e^{-h \cdot b_{\xi}(q,p)}$.

It follows that $\{\nu_x\}_{x\in X}$ is an $h$-dimensional Busemann density. By \cite[Theorem 4.10]{Kn1}, there exists exactly one Busemann density up to a scalar constant, which is realized by the Patterson-Sullivan measure. Thus $\{\nu_x\}_{x\in X}$ coincide with $\{\mu_x\}_{x\in X}$ up to a scalar constant.

\begin{proof}[Proof of Theorem \ref{function}]
Recall that $\tilde \mu_x$ is the normalized Patterson-Sullivan measure. By the above discussion, particularly $\nu_x(\partial X)=c(x)$ and $\frac{d\nu_{y}}{d \nu_{x}}(\xi)=e^{-h \cdot b_{\xi}(y,x)}$, we have
$$\frac{d\bar \mu_{y}}{d \bar\mu_{x}}(\xi)=\frac{c(x)}{c(y)}e^{-h \cdot b_{\xi}(y,x)}.$$
Then $c(y)=c(x)\int_{\pX}e^{-h \cdot b_{\xi}(y,x)} d\bar \mu_x(\xi).$

Recall that $y\mapsto b_{\xi}(y,x)$ is $C^2$ \cite[Theorem 1]{Es} (see also \cite[Section 2.3]{LS}), and moreover both $\nabla_yb_{\xi}(y,x)$ and $\triangle_yb_{\xi}(y,x)=tr U(y,\xi)$ depend both continuously on $\xi$.
It follows that $c$ is $C^1$.

If $c$ is constant, then $\int_{\pX}e^{-h \cdot b_{\xi}(y,x)} d\bar \mu_x(\xi)\equiv 1$. Taking Laplacian with respect to $y$ on both sides, we have
$$\int_{\pX}h(h-tr U(y,\xi))e^{-h \cdot b_{\xi}(y,x)} d\bar \mu_x(\xi)\equiv 0.$$
It follows that
$$h=\frac{\int_{\pX} tr U(y,\xi)e^{-h \cdot b_{\xi}(y,x)} d\bar \mu_x(\xi)}{\int_{\pX} e^{-h \cdot b_{\xi}(y,x)} d\bar \mu_x(\xi)}=\int_{\pX} tr U(y,\xi)d\bar \mu_y(\xi)$$
 for any $y\in X$.
\end{proof}

\section{Rigidity in dimension two}
Let $M$ be a rank one closed Riemannian surface of nonpositive curvature in this section.  A horocyclic flow is a continuous flow $h_s$ on $SM$ whose orbits are horocycles, i.e., for $v\in SM$, $\{h_sv: s\in \RR\}=W^s(v),$ where $W^s(v)$ is the strong unstable horocycle manifold of $v$ in $SM$.

\subsection{Unique ergodicity of horocycle flow}
We recall the recent results in \cite{Cl1} and \cite{Cl} on the unique ergodicity of the horocycle flow, which will be used in our arguments.

If $M$ has constant negative curvature, Furstenberg \cite{Fur} proved the unique ergodicity of the horocycle flow, which is extended to compact surfaces of variable negative curvature by Marcus \cite{Ma}.
To apply Marcus's method to surfaces of nonpositive curvature,  we need to define the horocycle flow using the so-called \emph{Margulis parametrization}. In \cite{Cl1} and \cite{Cl},  Margulis parametrization is supposed to be defined on certain subset $\Sigma$ of $SM$. There is another natural parametrization of the horocycle flow given by arc length of the horocycles, which clearly is well defined everywhere.
It is called the \emph{Lebesgue parametrization} and the Lebesgue horocycle flow is denoted by $h_s^L$.
Using the time change between $h_s^M|_\Sigma$ and $h_s^L|_\Sigma$, it is showed in \cite[Proposition 3.5]{Cl} that there is a bijection between finite invariant measures invariant under $h_s^M|_\Sigma$ and $h_s^L|_\Sigma$ respectively.

For example, let $\Sigma_0$ be the subset of $SM$ consisting of vectors whose horocycle contains a regular vector recurrent under the geodesic flow. Then the horocycle flow with Margulis parametrization $h_s^M$ is well defined on $\Sigma_0$. It showed in \cite[Theorem 3.6]{Cl1} that every finite Borel measure on $\Sigma_0$ invariant under horocycle flow $h_s^M$ is a constant multiple of $m|_{\Sigma_0}$, the Knieper measure $m$ restricted to $\Sigma_0$.
Thus by the above discussion, $h_s^L|_{\Sigma_0}$ is also uniquely ergodic.

From now on we assume that $M$ has no flat strips. In this case, we have $\Sigma=SM$, that is, $h_s^M$ is defined everywhere on $SM$ \cite[Proposition 4.1]{Cl}. Moreover, $h_s^M$ is uniquely ergodic on $SM$ and the unique finite Borel measure invariant under $h_s^M$ is the Knieper measure up to a constant multiple \cite[Theorem 4.2]{Cl}. It follows that $h_s^L$ is also uniquely ergodic on $SM$ \cite[Corollary 4.4]{Cl}. We are concerned with the unique probability measure $w^s$ invariant under the \emph{stable} horocycle flow.

Let $W^{wu}(v):=\{w\in SX: w^-=v^-\}$ be the weak unstable manifold of $v\in SX$, and $W^u(v):=\{w\in W^{wu}(v): b_{v^-}(\pi w, \pi v)=0\}$ the strong unstable horocycle manifold of $v$. Put $W^{ws}(v):=-W^{wu}(-v)$ and $W^{s}(v):=-W^{u}(-v)$.
Suppose that $W^{wu}(v)$ has no singular vectors. Then $W^{wu}(v)$ is a section of the (stable) horocycle flow $h_s^L$ in the sense of \cite[Lemma 3.2]{Cl}. By \cite[(2)]{Cl},
$$w^s(E)=\int_{W^{wu}(v)}\int_\RR \chi_E(h_s^L(u))dsd\mu_{W^{wu}(v)}(u)$$
where $\chi$ is the characteristic function, and $\mu_{W^{wu}(v)}$ is some Borel measure on $W^{wu}(v)$ which is in fact independent of the parametrization of the horocycle flow.

On the other hand, the unique invariant measure for $h_s^M$ is the Knieper measure $m$ \cite[Theorem 4.2]{Cl}, which can be expressed as
$$\int_{SX}\varphi dm=\int_\pX\int_\pX \int_\RR \varphi(\xi,\eta,t)e^{h\cdot \beta_{x}(\xi,\eta)}dtd\mu_x(\xi)d\mu_x(\eta)$$
since $X$ contains no flat strips.
Consider the canonical projection $$P=P_v: W^u(v)\to \pX, \quad P(w)=w^+,$$
then $$\mu_{W^{wu}(v)}(A)=\int \int_\RR \chi_A(\phi^tu)dt d\mu_{W^{u}(v)}(u)$$
where $\mu_{W^{u}(v)}(B):=\int_\pX \chi_{B}(P_v^{-1}\eta)e^{-hb_\eta(\pi v, x)}d\mu_x(\eta)$.

Thus $w^s$ is locally equivalent to the measure $ds \times dt \times d\mu_{W^u}$. If we disintegrate $w^s$ along $W^u$ foliation, the factor measure on a section $W^{ws}(v)$ is equivalent to $ds dt$, i.e., the Lebesgue measure $\text{Vol}$, and the conditional measures on the fiber $W^u(v)$ is equivalent to $P_v^{-1}\mu_x$.

\subsection{Integral formulas for topological entropy}
Recall that $M$ is a rank one closed Riemannian surface of nonpositive curvature without flat strips. Using the measure $w^s$ we can establish some formulas for topological entropy $h$ of the geodesic flow.

Let $B^s(v,R)$ denote the ball centered at $v$ of radius $R>0$ inside $W^s(v)$. In fact, it is just a curve. By the unique ergodicity of $h_s^L$, we have
\begin{lemma} \label{uniform}(Cf. \cite[Proposition 1.2]{Yue})
For any continuous $\varphi: SM\to \RR$,
$$\frac{1}{\text{Vol}(B^s(v,R))}\int_{B^s(v,R)}\varphi d\text{Vol}(y)\to \int_{SM}\varphi dw^s$$
as $R\to \infty$ uniformly in $v\in SM$.
\end{lemma}

For continuous $\varphi: SM\to \RR$, define $\varphi_x: X\to \RR$ by $\varphi_x(y)=\varphi(v(y))$ where $v(y)\in SX$ is the unique vector such that $c_{v(y)}(0)=y$ and $c_{v(y)}(t)=x$ for some $t\ge 0$.

Based on Lemma \ref{uniform}, we get the following proposition. The proof is the same as the one before \cite[Proposition 3.1]{Yue} (see also \cite{Kn4, Led}), and hence will be skipped. The basic idea here is that horospheres in $X$ can be approximated by geodesic spheres.
\begin{proposition}\label{uniform1}
For any continuous $\varphi: SM\to \RR$,
$$\frac{1}{s_R(x)}\int_{S(x,R)}\varphi_x(y) d\text{Vol}(y)\to \int_{SM}\varphi dw^s$$
as $R\to \infty$ uniformly in $x\in X$.
\end{proposition}

\begin{theorem}\label{formula}
Let $M$ be a rank one closed Riemannian surface of nonpositive curvature without flat strips. Then
\begin{enumerate}
  \item $h=\int_{SM} tr U(v) dw^s(v)$.
  \item $h^2=\int_{SM} -tr \dot{U}(v)+(tr U(v))^2 dw^s(v)$
\end{enumerate}
where $U(v)$ and $tr U(v)$ are the second fundamental form and the mean curvature of the horocycle $H_{\pi v}(v^+)$ at $\pi v$.
\end{theorem}
\begin{proof}
Consider the following function
$$G_x(R):=\frac{s_R(x)}{e^{hR}}=\frac{1}{e^{hR}}\int_{S(x,R)}d\text{Vol} (y).$$
Taking the derivatives, we have
\begin{equation*}
\begin{aligned}
G_x'(R)=&-hG_x(R)+\frac{1}{e^{hR}}\int_{S(x,R)}tr U_R(y) d\text{Vol} (y),\\
G_x''(R)=&-h^2G_x(R)-2hG_x'(R)+\\
&\frac{1}{e^{hR}}\int_{S(x,R)}-tr \dot U_R(y)+(tr U_R(y))^2 d\text{Vol} (y),\\
G_x'''(R)=&-h^3G_x(R)-3h^2G_x'(R)-3hG_x''(R)+\\
&\frac{1}{e^{hR}}\int_{S(x,R)}tr \ddot{U}_R(y)-3tr \dot U_R(y)tr U_R(y)+(tr U_R(y))^3 d\text{Vol} (y),
\end{aligned}
\end{equation*}
where $U_R(y)$ and $tr U_R(y)$ are the second fundamental form and the mean curvature of $S(x,R)$ at $y$.

Clearly, $tr U_R(y)\to tr U(v(y))$ as $R\to \infty$ uniformly. By Theorem \ref{margulis},
$$\lim_{R\to \infty}G_x(R)=\lim_{R\to \infty}\frac{s_R(x)}{e^{hR}}=c(x).$$
Combining with Proposition \ref{uniform1}, we have
\begin{equation}\label{e:deri}
\begin{aligned}
\lim_{R\to \infty}G_x'(R)=&-hc(x)+c(x)\int_{SM}tr U dw^s,\\
\lim_{R\to \infty}G_x''(R)=&-h^2c(x)-2h\lim_{R\to \infty}G_x'(R)+c(x)\int_{SM}-tr \dot U+(tr U)^2 dw^s,\\
\lim_{R\to \infty}G_x'''(R)=&-h^3c(x)-3h^2\lim_{R\to \infty}G_x'(R)-3h\lim_{R\to \infty}G_x''(R)+\\
&c(x)\int_{SM}tr \ddot{U}-3tr \dot U tr U+(tr U)^3 d w^s.
\end{aligned}
\end{equation}
Since $\lim_{R\to \infty}G_x'(R)$ exists and $\lim_{R\to \infty}G_x(R)$ is bounded, we have $\lim_{R\to \infty}G_x'(R)=0$. Similarly, considering the second and third derivative, we have
$$\lim_{R\to \infty}G_x''(R)=\lim_{R\to \infty}G_x'''(R)=0.$$
Plugging in \eqref{e:deri}, we have
\begin{enumerate}
  \item $h=\int_{SM} tr U(v) dw^s(v)$.
  \item $h^2=\int_{SM} -tr \dot{U}(v)+(tr U(v))^2 dw^s(v)$.
\end{enumerate}
\end{proof}

\subsection{Uniqueness of harmonic measure and rigidity}
We recall some facts from \cite{Yue} on the ergodic properties of foliations. Let $\mathcal{G}$ be any foliation on a compact Riemannian manifold $M$. A probability measure $\nu$ on $M$ is called \emph{harmonic} with respect to $\mathcal{G}$ if $\int_M \triangle^L fd\nu=0$ for any bounded measurable function $f$ on $M$ which is smooth in the leaf direction, where $\triangle^L$ denotes the Laplacian in the leaf direction.

A \emph{holonomy invariant measure} of the foliation $\mathcal{G}$ is a family of measures defined on each transversal of $\mathcal{G}$, which is invariant under all the canonical homeomorphisms of the holonomy pseudogroup \cite{Pl}. A measure is called \emph{completely invariant} with respect to $\mathcal{G}$ if it disintegrates to a constant function times the Lebesgue measure on the leaf, and the factor measure is a holonomy invariant measure on a transversal. By \cite{Ga}, a completely invariant measure must be a harmonic measure.

\begin{theorem}\label{unique}
Let $M$ be a rank one closed Riemannian surface of nonpositive curvature without flat strips. Then there is precisely one harmonic probability measure with respect to the strong stable horocycle foliation.
\end{theorem}
\begin{proof}
If $\dim M=2$, then the leaves of the strong stable horocycle foliation have polynomial volume growth. By \cite{Kai}, any harmonic measure must be completely invariant. By \cite{Cl}, there is a unique completely invariant measure $w^s$. As a consequence, $w^s$ is the unique harmonic measure.
\end{proof}

Recall that $\tilde \mu_x$ is a Borel measure on $S_xX$ (hence descending to $S_xM$) induced by the Patterson-Sullivan measure $\mu_x$ and let us assume that it is normalized, by a slight abuse of the notation. We have the following characterization of $w^s$.
\begin{proposition}\label{measure}
For any continuous $\varphi: SM\to \RR$, we have
$$C\int_{SM} \varphi dw^s=\int_Mc(x)\int_{S_xM}\varphi d\tilde \mu_x(v)d \text{Vol}(x)$$
where $C=\int_M c(x)d\text{Vol}(x)$.
\end{proposition}
\begin{proof}
The idea is to show the right hand side is a harmonic measure up to a normalization. Then the proposition follows from Theorem \ref{unique}. The proof is completely parallel to that of \cite[Proposition 4.1]{Yue} (see also \cite{Yue1,Led}), and hence is omitted.
\begin{comment}
By the discussion in Section 6.1, $w^s$ is locally equivalent to the measure $d\text{Vol} \times d\mu_{W^u}$.
Note that the Lebesgue measure $\text{Vol}$ on $W^{ws}(v)$ can be induced by the Lebesgue measure $\text{Vol}$ on $M$.
It remains to prove that the density function is given by $c(x)$.

Alternatively, we can prove the above proposition by showing that the measure on the right hand side is invariant under the horocycle flow, hence it coincides with $w^s$ by the unique ergodicity. Indeed, since the stable foliation is absolutely continuous inside each weak stable manifold, the action $h_s^L$ on any weak stable manifold $W^s(v)$ is absolutely continuous with respect the volume on  $W^s(v)$. We need argue that the inside integral is the density function.  Note that $P(A)=P(h_s^LA)$ where $P: SM\to \pX$ is the canonical projection $P(v)=v^+$. So the inside integral is unchanged under the action of $h_s^L$. On the other hand,It is easy to see that the density is given by $c(x)$.
\end{comment}
\end{proof}

\begin{proof}[Proof of Theorem \ref{rigidity}]
By Theorem \ref{formula},
$$h^2=\int_{SM} -tr \dot{U}(v)+(tr U(v))^2 dw^s(v).$$
By the Ricatti equation, in dimension two we have
$$-\dot U+U^2+K=0$$
where $K$ is the Gaussian curvature.
Since now $U$ is just a real number and hence $tr (U^2)=(tr U)^2$, using Proposition \ref{measure} and Gauss-Bonnet formula we have
\begin{equation*}
\begin{aligned}
h^2=&\int_{SM} -K dw^s=\frac{1}{C}\int_M -c(x)K(x)d \text{Vol}(x)\\
=&\int -Kd \text{Vol}/\text{Vol}(M)=-2\pi E /\text{Vol}(M),
\end{aligned}
\end{equation*}
where $E$ is the Euler characteristic of $M$. By Katok's result \cite[Theorem B]{Ka},
$h^2=-2\pi E /\text{Vol}(M)$ if and only if $M$ has constant negative curvature.
\end{proof}

\section{Flip invariance of the Patterson-Sullivan measure}
For each $x\in X$, denote by $\tilde \mu_x$ both the Borel probability measure on $S_xX$ and $\partial X$ given by the normalized Patterson-Sullivan measure.
For any continuous $\varphi: SM\to \RR$, define a measure $w^s$ by
$$C\int_{SM} \varphi dw^s:=\int_Mc(x)\int_{S_xM}\varphi d\tilde \mu_x(v)d \text{Vol}(x)$$
where $C=\int_M c(x)d\text{Vol}(x)$.

In view of the proof of Proposition \ref{measure}, $w^s$ is a harmonic measure associated to the strong stable foliation, though the uniqueness of harmonic measure is unknown in general. Without the uniqueness of harmonic measure, we can still obtain some rigidity results in this section.

\begin{proposition}\label{2formula}
For $\varphi\in C^1(SM)$, one has
$$\int_{SM}\dot\varphi+(h-tr U)\varphi dw^s=0.$$
\end{proposition}
\begin{proof}
Define a vector field on $M$ by
$$Y(y):=\int_{S_yM}\varphi X(v)d\mu_y(v)=\int_{S_xM}\varphi X(v)e^{-hb_v(y)}d\mu_x(v)$$
where $X$ is the geodesic spray. Since $div X=tr U$, one has
\begin{equation*}
\begin{aligned}
div|_{y=x}Y&=\int_{S_xM}div|_{y=x}\varphi X(v)e^{-hb_v(y)}d\mu_x(v)\\
&=\int_{S_xM}\dot\varphi+(h-tr U)\varphi d\mu_x.
\end{aligned}
\end{equation*}
Using Green's formula, we have $\int_{SM}\dot\varphi+(h-tr U)\varphi dw^s=0.$
\end{proof}

\begin{proposition}\label{coin}
If $w^s$ is $\phi^t$-invariant, then $M$ is locally symmetric.
\end{proposition}
\begin{proof}
If $w^s$ is $\phi^t$-invariant, by Proposition \ref{2formula}, we have
$$\int_{SM}(h-tr U)\varphi dw^s=0$$
for all  $\varphi\in C^1(SM,\RR)$. It follows that $tr U\equiv h$, i.e., $M$ is asymptotically harmonic. By \cite[Theorem 1.2]{Zi} (see also \cite[Proposition 2.2]{LS}), $M$ is locally symmetric.
\end{proof}

For manifolds of nonpositive curvature, not every pair of $\eta\neq \xi$ in $\partial X$ can be connected by a geodesic. A point $\xi\in \partial X$ is called \emph{hyperbolic} if for any $\eta\neq \xi$ in $\partial X$, there exists a rank one geodesic joining $\eta$ to $\xi$. The set of hyperbolic points is dense in $\partial X$ (see \cite[Lemma 3.4]{LS}).
\begin{lemma}\label{liou}
If for all $x\in M$, $\tilde\mu_x$ is flip invariant, then the Knieper measure $\lm$ coincides with the Liouville measure $\text{Leb}$ on $SM$.
\end{lemma}

\begin{proof}
First we lift every measure to the universal cover $X$ and show that for all $x\in X$, $\frac{d\tilde \mu_x}{d \text{Leb}}_x$ is finite every where on $S_xX$.
We still denote the measures $f_x \tilde\mu_x$ and $f_x \text{Leb}_x$ on $\partial X$ by $\tilde\mu_x$ and $\text{Leb}_x$ for simplicity.

Assume that there exists some $\xi\in \partial X$ such that
\begin{equation}\label{e:ratio}
\begin{aligned}
\limsup_{\e\to 0}\frac{\tilde \mu_x(D_x(\xi,\e))}{\text{Leb}_x(D_x(\xi,\e))}=0
\end{aligned}
\end{equation}
where $D_x(\xi,\e):=\{\eta\in \partial X: \angle_x(\xi,\eta)\le \e\}$. Take $\e>0$. For any $\rho>0$ small enough, choose a hyperbolic $\xi'\in \partial X$ close to $\xi$ such that
\begin{equation}\label{e:ratio1}
\begin{aligned}
D_x(\xi', (1-\rho)\e)\subset D_x(\xi,\e)\subset D_x(\xi', (1+\rho)\e).
\end{aligned}
\end{equation}
We can choose some constant $C_1>1$ independent of $\e$ and $\rho$ such that
\begin{equation}\label{e:ratio2}
\begin{aligned}
\text{Leb}_x(D_x(\xi,\e))\le \text{Leb}_x(D_x(\xi', (1+\rho)\e))\le C_1\text{Leb}_x(D_x(\xi', (1-\rho)\e)).
\end{aligned}
\end{equation}
It follows from \eqref{e:ratio},\eqref{e:ratio1} and \eqref{e:ratio2} that
\begin{equation}\label{e:ratio3}
\begin{aligned}
\frac{\tilde \mu_x(D_x(\xi',(1-\rho)\e))}{\text{Leb}_x(D_x(\xi',(1-\rho)\e))}\le \frac{C_1\tilde \mu_x(D_x(\xi,\e))}{\text{Leb}_x(D_x(\xi,\e))}.
\end{aligned}
\end{equation}
Then for any $\eta\in \partial X$, there exists a geodesic $c_{\xi'\eta}$ connecting $\xi'$ and $\eta$. Take a point $y\in c_{\xi'\eta}$. Due to the flip invariance,
\begin{equation*}\label{e:flip}
\begin{aligned}
\frac{\tilde \mu_y(D_y(\xi',(1-\rho)\e))}{\text{Leb}_y(D_y(\xi',(1-\rho)\e))}=\frac{\tilde \mu_y(D_y(\eta,(1-\rho)\e))}{\text{Leb}_y(D_y(\eta,(1-\rho)\e))}.
\end{aligned}
\end{equation*}

Since the measures $\tilde \mu_y$ and $\tilde \mu_x$ (resp. $\text{Leb}_y$ and $\text{Leb}_x$) are equivalent with positive Radon-Nikodym derivative, we have
by \eqref{e:ratio3},
\begin{equation*}\label{e:ratio4}
\begin{aligned}
\frac{\tilde \mu_y(D_x(\xi',(1-\rho)\e))}{\text{Leb}_y(D_x(\xi',(1-\rho)\e))}\le \frac{C_2\tilde \mu_x(D_x(\xi,\e))}{\text{Leb}_x(D_x(\xi,\e))}
%\limsup_{\e\to 0}\frac{\tilde \mu_y(D_x(\xi',\e))}{\text{Leb}_y(D_x(\xi',\e))}=0.
\end{aligned}
\end{equation*}
for some $C_2>1$.
Then by flip invariance,
\begin{equation*}\label{e:ratio6}
\begin{aligned}
\frac{\tilde \mu_y(U_\e(\eta))}{\text{Leb}_y(U_\e(\eta))}\le \frac{C_2\tilde \mu_x(D_x(\xi,\e))}{\text{Leb}_x(D_x(\xi,\e))}
\end{aligned}
\end{equation*}
where $U_\e(\eta)$ is the image of $D_x(\xi',(1-\rho)\e)$ under the flip map. Use again that $\tilde \mu_y$ and $\tilde \mu_x$ (resp. $\text{Leb}_y$ and $\text{Leb}_x$) are equivalent with positive Radon-Nikodym derivative, we get for some $C_3>1$
\begin{equation*}\label{e:ratio6}
\begin{aligned}
\frac{\tilde \mu_x(U_\e(\eta))}{\text{Leb}_x(U_\e(\eta))}\le \frac{C_3\tilde \mu_x(D_x(\xi,\e))}{\text{Leb}_x(D_x(\xi,\e))}.
\end{aligned}
\end{equation*}
As $U_\e(\eta)$ shrinks to $\{\eta\}$ as $\e\to 0$, we see the Radon-Nikodym derivatives $\frac{d\tilde \mu_x}{d\text{Leb}_x}$ is also zero at any $\eta\in \pX$.

Similarly, if
\begin{equation*}
\begin{aligned}
\limsup_{\e\to 0}\frac{\tilde \mu_x(D_x(\xi,\e))}{\text{Leb}_x(D_x(\xi,\e))}=\infty
\end{aligned}
\end{equation*}
for some $\xi\in \partial X$, the Radon-Nikodym derivatives $\frac{d\tilde \mu_x}{d\text{Leb}_x}$ is also infinity at any $\eta\in \pX$.

Since both $\tilde \mu_x$ and $\text{Leb}_x$ have finite total mass, their Radon-Nikodym derivatives must be finite somewhere and hence everywhere. Thus the Liouville measure $\text{Leb}$ is equivalent to the Knieper measure. As the Knieper measure is ergodic, the two measures coincide.
\end{proof}

\begin{lemma}\label{cons}
If for all $x\in M$ $\tilde\mu_x$ is flip invariant, then the Margulis function $c(x)$ is constant.
\end{lemma}
\begin{proof}
Any $\varphi\in C^2(M,\RR)$ can be lifted to a function on $SM$ which we still denote by $\varphi$. Since any weak unstable manifold is diffeomorphic to $X$, we have $\bigtriangleup^{cs} \varphi=\triangle \varphi$ where $\triangle$ is the Laplacian along $X$ and $\bigtriangleup^{cs}$ is the Laplacian along the weak stable foliation. By \cite[Lemma 5.1]{Yue3}, $\triangle^{cs} \varphi=\bigtriangleup^s\varphi+\ddot{\varphi}-tr U \dot\varphi$.
Then by Definition of $w^s$ and Proposition \ref{2formula},
\begin{equation*}
\begin{aligned}
&\int_M \bigtriangleup \varphi c(x)d\text{Leb}=C\int \bigtriangleup \varphi dw^s\\
=&C\int_{SM}(\bigtriangleup^s\varphi+\ddot{\varphi}-tr U \dot\varphi)dw^s\\
=&C\left(\int_{SM}\bigtriangleup^s\varphi dw^s+\int_{SM}\ddot{\varphi}+(h-tr U) \dot\varphi dw^s-\int_{SM}h\dot\varphi dw^s\right)\\
=&-h\int_{M}c(x)d\text{Leb}(x)\int \dot\varphi(x,\xi) d\tilde\mu_x(\xi).
\end{aligned}
\end{equation*}
Since $d\tilde\mu_x(\xi)=d\tilde\mu_x(-\xi)$ and $\dot\varphi(x,\xi)=-\dot\varphi(x,-\xi)$, we have $$\int_M \bigtriangleup \varphi c(x)d \text{Leb}=0$$
for any $\varphi\in C^2(M,\RR)$. So $c(x)$ must be constant.
\end{proof}

\begin{proof}[Proof of Theorem \ref{flip}]
By the construction, the Knieper measure $\lm$ is flip invariant. By the flip invariance of the partition $\{S_xM\}_{x\in M}$ and the uniqueness of conditional measures, we see that $\bar\mu_x$ is flip invariant for $\lm\ae x\in M$. It follows that the normalized Patterson-Sullivan measures $\tilde\mu_x$ is flip invariant for $\lm\ae x\in M$.

We claim that for all $x\in M$ $\tilde\mu_x$ is flip invariant. Indeed, as
$$\frac{d\tilde \mu_{y}}{d \tilde\mu_{x}}(\xi)=\frac{c(x)}{c(y)}e^{-h \cdot b_{\xi}(y,x)},$$
$x\mapsto \tilde\mu_{x}$ is continuous. Let $x_k,x\in X$ and $x_k\to x$ as $k\to \infty$. Then for each Borel subset $A\subset \partial X$ and its geodesic reflection $-A(x_k)$ (resp. $-A(x)$) with respect to $x_k$ (resp. $x$),
$$\tilde\mu_x(A)=\lim_{k\to \infty} \tilde\mu_{x_k}(A)=\lim_{k\to \infty} \tilde\mu_{x_k}(-A(x_k))=\tilde\mu_x(-A(x)).$$
The claim follows.

By Lemma \ref{liou}, the Knieper measure $\lm$ coincides with the Liouville measure, and thus $\lm$ projects to the Riemannian volume on $M$. By assumption, the conditional measures $\bar\mu_x$ coincides with $\tilde\mu$. Moreover, by Lemma \ref{cons}, $c(x)$ is constant. Consequently, we see from definition that $w^s$ coincides with the Knieper measure $\lm$, and hence it is $\phi^t$-invariant. By Proposition \ref{coin}, $M$ is locally symmetric.
\end{proof}

\ \
\\[-2mm]
\textbf{Acknowledgement.}
The author would like to thank the referees for valuable suggestions. This work is supported by NSFC Nos. 12071474 and 11701559, and Fundamental Research Funds for the Central Universities No. 20720210038.

\section{Appendix: Manifolds without focal/conjugate points}
In this appendix, we discuss the proof of Theorem A', which is an extension of Theorem \ref{margulis} to all manifolds without focal points, and a class of manifolds without conjugate points.
\begin{definition}
Let $c$ be a geodesic on $(M,g)$.
\begin{enumerate}
  \item A pair of distinct points $p=c(t_{1})$ and $q=c(t_{2})$ are called \emph{focal} if there is a Jacobi field $J$ along $c$ such that $J(t_{1})=0$, $J'(t_{1})\neq 0$ and $\frac{d}{dt}\mid_{t=t_{2}}\| J(t)\|^{2}=0$;
  \item $p=c(t_{1})$ and $q=c(t_{2})$ are called \emph{conjugate} if there is a nontrivial Jacobi field $J$ along $c$ such that $J(t_{1})=0=J(t_{2})$.
\end{enumerate}
A compact Riemannian manifold $(M,g)$ is called a manifold \emph{without focal points/without conjugate points} if there is no focal points/conjugate points on any geodesic in $(M,g)$.
\end{definition}
By definition, if a manifold has no focal points then it has no conjugate points. All manifolds of nonpositive curvature always have no focal points.

The uniqueness of MME for manifolds without focal points is obtained in \cite{LWW,CKP1,CKP2}. In \cite{CKW1}, the authors proved the uniqueness of MME for the \emph{class $\mathcal{H}$} of manifolds without conjugate points that satisfy:
\begin{enumerate}
  \item There exists a Riemannian metric $g_0$ on $M$ for which all sectional curvatures are negative;
  \item The uniform visibility axiom (see Definition \ref{vis} below) is satisfied;
  \item The fundamental group $\pi_1(M)$ is residually finite: the intersection of its finite index subgroups is trivial;
  \item There exists $h_0 < h$ such that any ergodic invariant Borel probability measure $\mu$ on $SM$ with entropy strictly greater than $h_0$ is almost expansive (cf. \cite[Definition 2.8]{CKW2}) .
\end{enumerate}
All surfaces of genus at least $2$ without conjugate points belong to the class $\mathcal{H}$.

\begin{definition}(Cf. \cite[Definition 2.1]{CKW2})\label{vis}
A simply connected Riemannian manifold $X$ is a \emph{uniform visibility manifold} if for every $\e>0$ there exists $L=L(\e) > 0$ such that whenever a geodesic $c: [a, b]\to X$ stays at distance at least $L$ from some point $p \in X$,
then the angle sustained by $c$ at $p$ is less than $\e$, that is,
$$\angle_p(c) := \sup_{a\le s,t\le b}\angle_p(c(s), c(t)) < \e.$$
If $M$ is a Riemannian manifold without conjugate points whose universal cover $X$ is a uniform visibility manifold, then we say that $M$ is a uniform visibility manifold.
\end{definition}

\begin{definition}(Cf. \cite[Definition 2.2]{CKW2})
The manifold $(M, g)$ has the \emph{divergence property} if given any geodesics $c_1\neq c_2$ with $c_1(0)=c_2(0)$ in the universal cover, we have $\lim_{t\to \infty} d(c_1(t), c_2(t))=\infty$. Manifolds without focal points has the divergence property.
\end{definition}
The uniform visibility property implies the divergence property. All manifolds without focal points have the divergence property.

The proof of Theorem A' is analogous to that of Theorem \ref{margulis} with minor modifications. We skip the details and just sketch main steps where modifications are needed.

The local product flow boxes are constructed in \cite[Section 2.4]{Wu} near regular vectors in the no focal points case, and in \cite[Section 3.2]{CKW2} near expansive vectors in the case of no conjugate points. We need modify the time interval from $[0,\a]$ to $[-\a,-\a]$, so that Lemma \ref{nhd} still holds.

We must establish corresponding versions of $\pi$-convergence theorem \ref{piconvergence}. For manifolds with no focal points, a weak $\pi$-convergence theorem is established in \cite[Theorem 3.9]{Wu}, and now we need to replace $\bF_\theta$ by $V^+$, and $\bF_\rho$ by $(B_{-\rho}N)^+$ in that theorem. The proof can be modified accordingly. In the case of no conjugate points, \cite[Lemma 4.9]{CKW2} should be rephrased and reproved accordingly.

In both the proofs of Lemmas \ref{intersect2} and \ref{nonuniform}, the comparison theorem is used to show the following: Let $p\in X$, then given any $a>0,\e>0$, there exists $T>0$ such that for any $t\geq T$ and $v,w\in S_pX$, $d(\phi^tv,\phi^tw)\leq \e$ implies that $\angle(v,w)<a$.  Now we can not use the comparison theorem anymore.
\begin{enumerate}
  \item For manifolds in class $\mathcal{H}$, it is a direct consequence of uniform visibility property. Indeed, if $T$ is large enough, then by the triangle inequality, the geodesic connecting $\phi^tv$ and $\phi^tw$ stays at distance at least $L(a)$ from $p$. Thus $\angle(v,w)<a$.
  \item For rank one manifolds without focal points, assume the contrary.  Then there exist $t_n\to \infty$ and $v_n,w_n\in S_pX$ such that
  $$d(\phi^{t_n}v_n,\phi^{t_n}w_n)\leq \e \text{\ and\ }\angle(v_n,w_n)\ge a.$$
  By taking a subsequence, we can assume without loss of generality that $v_n\to v, w_n\to w$ for some $v,w\in S_pX$. Then $\angle(v,w)\ge a$. Take any $t>0$. Choose $n$ large enough such that $t_n>t$ and $d(\phi^tv,\phi^tw)\le  d(\phi^tv_n,\phi^tw_n)+\e$. By monotonicity, $d(\phi^tv_n,\phi^tw_n)\leq d(\phi^{t_n}v_n,\phi^{t_n}w_n) \le \e$. It follows that $d(\phi^tv,\phi^tw)\leq 2\e$ for any $t>0$, which contradicts to the divergence property.
\end{enumerate}
So we also have these lemmas in both no focal/conjugate points cases.

Finally, let us comment on Lemma \ref{singular}, which can be extended to the no focal points case without any modification. In the case of no conjugate points, instead of singular vectors we need consider vectors which do not lie in a countable union of flow boxes near expansive vectors. More precisely, there exist countably many expansive vectors $w_1, w_2, \cdots$ such that $S_xX\cap \mathcal{E} \subset \cup_{i=1}^\infty\text{int} B_{\theta_i}^\a(w_i)$, where $\mathcal E$ is the expansive set. See \cite[(2.11)]{CKW1} for definition of expansive vectors and expansive set. The vectors outside of these flow boxes form a subset $S$ which is closed and $\phi^t$-invariant. Moreover, $S\cap \mathcal E=\emptyset$. Since the unique MME $\lm$ gives full weight to $\mathcal E$ (cf. \cite[Theorem 5.6]{CKW1}), we know $\lm(S)=0$. It follows that $h_\top(S)<h$ and thus Lemma \ref{singular} can be proved similarly.

\end{document}